\documentclass[a4paper,twopage,reqno,11pt]{amsart}
\usepackage[top=30mm,right=30mm,bottom=30mm,left=30mm]{geometry}

\usepackage{tabularx}
\usepackage{graphicx}
\usepackage{amsmath}
\usepackage{amssymb}
\usepackage{amsfonts}
\usepackage{amsthm}
\usepackage{amstext}
\usepackage{amsbsy}
\usepackage{amsopn}
\usepackage{amscd}
\usepackage{enumerate}
\usepackage{color}
\usepackage[colorlinks]{hyperref}
\usepackage[hyperpageref]{backref}
\usepackage{multirow}
\usepackage{longtable}
\usepackage{float}
\usepackage{lscape}
\usepackage{pdflscape}
\usepackage{url}

\newtheorem{theorem}{Theorem}[section]
\newtheorem{corollary}[theorem]{Corollary}
\newtheorem{lemma}[theorem]{Lemma}
\newtheorem{proposition}[theorem]{Proposition}

\theoremstyle{definition}

\newtheorem{example}[theorem]{Example}
\newtheorem{remark}[theorem]{Remark}

\numberwithin{equation}{section}

\newcommand{\GO}{\mathrm{GO}}
\newcommand{\GU}{\mathrm{GU}}
\newcommand{\G}{\mathrm{G}}

\newcommand{\GL}{\mathrm{GL}}
\newcommand{\SL}{\mathrm{SL}}

\newcommand{\Sp}{\mathrm{Sp}}

\newcommand{\SO}{\mathrm{SO}}
\newcommand{\SU}{\mathrm{SU}}

\renewcommand{\O}{\mathrm{O}}

\newcommand{\PSL}{\mathrm{PSL}}
\newcommand{\PSU}{\mathrm{PSU}}
\newcommand{\U}{\mathrm{U}}
\newcommand{\PGU}{\mathrm{PGU}}

\newcommand{\PSp}{\mathrm{PSp}}

\newcommand{\PGL}{\mathrm{PGL}}
\newcommand{\PGaL}{\mathrm{P\Gamma L}}

\newcommand{\AGaL}{\mathrm{A\Gamma L}}
\newcommand{\POm}{\mathrm{P \Omega}}

\newcommand{\A}{\mathrm{A}}
\newcommand{\W}{\mathrm{W}}
\newcommand{\E}{\mathrm{E}}

\renewcommand{\S}{\mathrm{S}}
\newcommand{\Q}{\mathrm{Q}}

\newcommand{\D}{\mathrm{D}}

\newcommand{\Aut}{\mathrm{Aut}}

\newcommand{\Out}{\mathrm{Out}}

\newcommand{\PG}{\mathrm{PG}}

\newcommand{\Fix}{\mathrm{Fix}}

\newcommand{\B}{\mathrm{B}}

\newcommand{\M}{\mathrm{M}}

\newcommand{\Fbb}{\mathbb{F}}

\newcommand{\Dmc}{\mathcal{D}}
\newcommand{\Bmc}{\mathcal{B}}

\newcommand{\Pmc}{\mathcal{P}}
\newcommand{\Cmc}{\mathcal{C}}
\newcommand{\Smc}{\mathcal{S}}

\newcommand{\e}{\epsilon}

\renewcommand{\leq}{\leqslant}
\renewcommand{\geq}{\geqslant}

\newcommand{\Binom}[2]{\genfrac{[}{]}{0pt}{}{#1}{#2}}

\begin{document}
\title[Flag-transitive $2$-designs]{A classification of flag-transitive block designs}

\author[S.H. Alavi]{Seyed Hassan Alavi}%
\thanks{Corresponding author: S.H. Alavi}
\address{Seyed Hassan Alavi, Department of Mathematics, Faculty of Science, Bu-Ali Sina University, Hamedan, Iran.
}%
\email{alavi.s.hassan@basu.ac.ir and alavi.s.hassan@gmail.com (G-mail is preferred)}

\author[A. Daneshkhah]{Ashraf Daneshkhah}%
\address{Ashraf Daneshkhah, Department of Mathematics, Faculty of Science, Bu-Ali Sina University, Hamedan, Iran.
}%
\email{adanesh@basu.ac.ir}
\author[F. Mouseli]{Fatemeh Mouseli}%
\address{Fatemeh Mouseli, Department of Mathematics, Faculty of Science, Bu-Ali Sina University, Hamedan, Iran.}%
\email{f.mooseli@gmail.com}

\subjclass[]{05B05; 05B25; 20B25}%
\keywords{$2$-design; flag-transitive; automorphism group;   primitive group; large subgroup}
\date{\today}%

\begin{abstract}
  In this article, we investigate $2$-$(v,k,\lambda)$ designs with $\gcd(r,\lambda)=1$ admitting flag-transitive automorphism groups $G$. We prove that if $G$ is an almost simple group, then such a design belongs to one of the seven infinite families of $2$-designs or it is one of the eleven well-known examples. We describe all these examples  of designs. We, in particular, prove that if $\Dmc$ is a symmetric $(v,k,\lambda)$ design with $\gcd(k,\lambda)=1$ admitting a flag-transitive automorphism group $G$, then either $G\leq \AGaL_{1}(q)$ for some odd prime power $q$, or $\Dmc$ is a projective space or the unique Hadamard design with parameters  $(11,5,2)$.
\end{abstract}

\maketitle

\section{Introduction}\label{sec:intro}

A $2$-$(v,k,\lambda)$ design $\Dmc$ is a pair $(\Pmc,\Bmc)$ with a set $\Pmc$ of $v$ points and a set $\Bmc$ of $b$ blocks such that each block is a $k$-subset of $\Pmc$ and each two distinct points are contained in $\lambda$ blocks. The \emph{replication number} $r$ of $\Dmc$ is the number of blocks incident with a given point. A \emph{symmetric} design is a $2$-design with the same number of points and blocks, that is to say, $v=b$. An \emph{automorphism} of $\Dmc$ is a permutation on $\Pmc$ which maps blocks to blocks and preserving the incidence. The \emph{full automorphism} group $\Aut(\Dmc)$ of $\Dmc$ is the group consisting of all automorphisms of $\Dmc$. A \emph{flag} of $\Dmc$ is a point-block pair $(\alpha, B)$ such that $\alpha \in B$. For $G\leq \Aut(\Dmc)$, $G$ is called \emph{flag-transitive} if $G$ acts transitively on the set of flags. The group $G$ is said to be \emph{point-primitive} if $G$ acts primitively on $\Pmc$. A group $G$ is said to be \emph{almost simple} with socle $X$ if $X\unlhd G\leq \Aut(X)$, where $X$ is a nonabelian simple group. For finite simple groups of Lie type, we adopt the standard notation as in \cite{b:Atlas,b:KL-90}, and in particular, we use the notation described in Section~\ref{sec:defn} and therein references. 

The main aim of this paper is to study $2$-designs with flag-transitive automorphism groups.
In 1988, Zieschang~\cite{a:Zieschang-88} proved that if an automorphism group $G$ of a $2$-design with $\gcd(r,\lambda)=1$ is flag-transitive, then $G$ is a point-primitive group of almost simple or affine type. Such designs admitting an almost simple automorphism group with socle being an alternating group, a sporadic simple group, a projective special unitary group or a finite simple exceptional group have been studied in~\cite{a:A-Exp-CP,a:ABD-Un-CP,a:Zhou-sym-sporadic,a:Zhan-CP-nonsym-sprodic,a:Zhou-CP-nonsym-alt,a:Zhuo-CP-sym-alt}. The present paper is devoted to determining all possible $2$-designs with $\gcd(r,\lambda)=1$ admitting a flag-transitive almost simple automorphism group $G$ with socle $X$ being a nonabelian finite simple group. The examples of such designs are given in Section \ref{sec:examples}, and our main result is Theorem~\ref{thm:main} below:

\begin{theorem}\label{thm:main}
	Let $\Dmc=(\Pmc,\Bmc)$ be a nontrivial $2$-$(v,k,\lambda)$ design with $r$ being coprime to $\lambda$, and let $\alpha$ be a
	point of $\Dmc$. If $G$ is a flag-transitive automorphism group of $\Dmc$ with socle $X$ being a nonabelian finite simple group and $H=G_{\alpha}$, then $\Dmc$ and $G$ are as in {\rm Examples~\ref{ex:proj-space}-\ref{ex:other}}, or $X=\PSL_n(q)$ with $n\geq 3$ and $H\cap X\cong \,^{\hat{}}[q^{n-1}]{:}\SL_{n-1}(q){\cdot} (q-1)$ is a parabolic subgroup and $\Dmc$ is a  $2$-$((q^{n}-1)/(q-1),q,q-1)$ with $\gcd(n-1,q-1)=1$, where $\Pmc$ is the point set of $\PG_{n-1}(q)$ and $\Bmc=(B\setminus\{\alpha\})^{G}$ with $B$ a line in $\PG_{n-1}(q)$ and $\alpha$ any point in $B$. 
\end{theorem}


Our main result Theorem~\ref{thm:main} together with the recent studies of Biliotti, Francot and Montinaro \cite{a:Biliotti-CP-sol-HA,a:Biliotti-CP-nonsol-HA,a:Biliotti-CP-sym-affine} on affine type automorphism groups will complete the classification of $2$-designs with $\gcd(r,\lambda)=1$ admitting a flag-transitive automorphism group except for one dimensional affine type automorphism groups $G\leq \AGaL_{1}(q)$. 

As another consequence of Theorem~\ref{thm:main}, we are able to extend our result in \cite{a:ABCD-PrimeRep} and to obtain all $2$-designs with prime replication number admitting flag-transitive automorphism groups of almost simple type.

\begin{corollary}\label{cor:rp}
	Let $\Dmc$ be a nontrivial $2$-design with parameters $(v,k,\lambda)$ and prime replication number $r$, and let $\alpha$ be a point of $\Dmc$. If $G$ is a flag-transitive automorphism group of $\Dmc$ of almost simple type with socle $X$, then $\Dmc$ and $G$ are as in lines {\rm 1-3, 7, 11, 13} of {\rm Table~\ref{tbl:main}} or one of the following holds:
	\begin{enumerate}[{\quad \rm (a)}]
		\item $\Dmc$ is the projective space and $X=\PSL_{n}(q)$ for $n\geq 3$ with $r=(q^{n-1}-1)/(q-1)$ prime {\rm (Example~\ref{ex:proj-space})};
		\item $\Dmc$ is the Witt-Bose-Shrikhande space $\W(2^a)$ and $X=\PSL_{2}(2^{a})$ with $r=2^{a}+1$ Fermat prime {\rm (Example~\ref{ex:witt})}.
	\end{enumerate}
\end{corollary}

Symmetric designs admitting flag-transitive automorphism groups are of most interest.
Such designs with $\gcd(k,\lambda)=1$ admitting a flag-transitive almost simple automorphism group whose socle is an alternating group, a sporadic simple group, a projective special unitary group or a finite simple exceptional group have been studied in~\cite{a:A-Exp-CP,a:ABD-Un-CP,a:Zhou-sym-sporadic,a:Zhuo-CP-sym-alt}. Biliotti and Montinaro \cite{a:Biliotti-CP-sym-affine} studied the  affine type automorphism groups $G$ of symmetric designs with $\gcd(k,\lambda)=1$, and proved that $G\leq \AGaL_{1}(q)$, for some odd prime power $q$. Therefore, as a main consequence of Theorem~\ref{thm:main}, excluding flag-transitive  $1$-dimensional affine  automorphism groups, we prove that projective spaces and the unique Hadamard design with parameters  $(11,5,2)$ are the only examples of nontrivial symmetric $(v,k,\lambda)$ designs with $\gcd(k,\lambda)=1$ admitting a flag-transitive automorphism group:

\begin{corollary}\label{cor:main}
	Let $\Dmc$ be a nontrivial symmetric $(v,k,\lambda)$ design with $\gcd(k,\lambda)=1$ admitting a flag-transitive automorphism group $G$, then either $v=p^{d}$ is odd and $G\leq \AGaL_{1}(p^{d})$ is point-primitive and block-primitive, or $\Dmc$ is a projective space $\PG_{n-1}(q)$ as in {\rm Example~\ref{ex:proj-space}} or the unique Hadamard design with parameters  $(11,5,2)$ as in line $6$ of {\rm Table~\ref{tbl:main}}.
\end{corollary}

\subsection{Outline of the proof}\label{sec:outline}

In order to prove Theorem~\ref{thm:main} and its corollaries in Section~\ref{sec:proof}, we observe that the group $G$ is point-primitive~\cite{a:Zieschang-88}, or equivalently, the point-stabiliser $H=G_\alpha$ is maximal in $G$. Moreover, we only need to focus on the case where $X$ is a nonabelian finite simple classical group but not a projective special unitary group ~\cite{a:A-Exp-CP,a:ABD-Un-CP,a:Zhou-sym-sporadic,a:Zhan-CP-nonsym-sprodic,a:Zhou-CP-nonsym-alt,a:Zhuo-CP-sym-alt}. By Aschbacher's Theorem~\cite{a:Aschbacher-84}, the maximal subgroup $H$ belongs to  one of the eight geometric families $\Cmc_{i}$ ($i=1,\ldots ,8$) of subgroups of $G$, or it is in the family $\Smc$ of almost simple subgroups with some irreducibility conditions. We then obtain the subgroups $H$ satisfying $|H \cap X|_{p} < |\Out(X)|$ in  Lemma~\ref{lem:coprime}, and together with the list of large maximal subgroups of almost simple groups satisfying $|X|\leq |H\cap X|^3$ recorded in \cite{a:AB-Large-15}, we obtain possible candidates for subgroups $H$. We then analyse these possible cases and prove the main results. In particular, we use detailed information of the subdegrees of the primitive actions of finite simple classical groups. We note here that for computational arguments, we use the software \textsf{GAP} \cite{GAP4}.

\subsection{Definitions and notation}\label{sec:defn}

All groups and incidence structures in this paper are finite. Symmetric and alternating groups on $n$ letters are denoted by $\S_{n}$ and $\A_{n}$, respectively.  We write ``$n$'' for the cyclic group of order $n$. For finite simple groups of Lie type, we adopt the standard notation as in \cite{b:Atlas}, and in particular, we use the following notation to denote the finite simple classical groups:
\begin{align*}
	&\PSL_{n}(q), \text{ for } n\geq 2 \text{ and } (n,q)\neq (2,2), (2,3),\\
	&\PSU_{n}(q), \text{ for } n\geq 3 \text{ and } (n,q)\neq (3,2),\\
	&\PSp_{2m}(q), \text{ for }n=2 m \geq 4 \text{ and } (m,q)\neq (2,2),\\
	&\POm_{2m+1}(q), \text{ for } n=2m+1\geq 7 \text{ and } q \text{ odd},\\
	&\POm_{2m}^{\pm}(q),  \text{ for } n=2m\geq 8.
\end{align*}
In this manner, the only repetitions are
\begin{align*}
	\nonumber  &\PSL_{2}(4)\cong \PSL_{2}(5)\cong \A_{5}, & &
	\PSL_{2}(7)\cong \PSL_{3}(2), & &
	\PSL_{2}(9)\cong \A_{6},\\
	&\PSL_{4}(2)\cong \A_{8}, & &\PSp_{4}(3)\cong \PSU_{4}(2).
\end{align*}

Recall that a $2$-design $\Dmc$ with parameters  $(v,k,\lambda)$ is a pair $(\Pmc,\Bmc)$, where $\Pmc$ is a set of $v$ points and $\Bmc$ is a set of $b$ blocks such that each block is a $k$-subset of $\Pmc$ and each two distinct points are contained in $\lambda$ blocks. We say that $\Dmc$ is nontrivial if $2 < k < v-1$.
If $b=v$, or equivalently $r=k$, the design $\Dmc$ is called a symmetric design, otherwise it is called nonsymmetric, where the replication number $r$ is the number of blocks incident with a given point. For a nonsymmetric design, we always have $b>v$ and $r>k$.
Further notation and definitions in both design theory and group theory are standard and can be found, for example, in~\cite{b:Atlas,b:Dixon,t:Kleidman,b:lander}.

\section{Examples}\label{sec:examples}

In this section, we provide some examples of $2$-designs with $\gcd(r,\lambda)=1$ admitting a flag-transitive automorphism almost simple group with socle $X$. The $2$-designs in Example~\ref{ex:proj-space}-\ref{ex:ree-designs} arose naturally in the study of linear spaces \cite{a:BDD-1988,a:Kantor-87-Odd,a:Saxl2002}, Examples~\ref{ex:ree-designs} and~\ref{ex:suzuki} appear in \cite{a:A-Exp-CP} when the socle $X$ of $G$ is a finite simple exceptional group, and the $2$-designs in Example~\ref{ex:other} are obtained in \cite{a:ABD-Un-CP,a:Zhou-sym-sporadic,a:Zhan-CP-nonsym-sprodic,a:Zhou-CP-nonsym-alt,a:Zhuo-CP-sym-alt}.   We note here that the examples of symmetric designs occur only in Examples~\ref{ex:proj-space} and~\ref{ex:other}.

\begin{example}\label{ex:proj-space}
	The projective spaces $\PG_{n-1}(q)$ with parameters  $((q^{n}-1)/(q-1),(q^{n-1}-1)/(q-1),(q^{n-2}-1)/(q-1))$ for $n\geq 3$ is a well-known example of flag-transitive symmetric designs. Any group $G$ with $\PSL_{n}(q) \leq G \leq \PGaL_{n}(q)$ acts flag-transitively on $\PG_{n-1}(q)$. If $n=3$, then we have the \emph{Desargusian plane} with parameters $(q^{2}+q+1,q+1,1)$ which is a projective plane. We remark that there is one additional example with $G=\A_7$ on $\Dmc=\PG_{3}(2)$, see \cite{a:BDD-1988,a:Kantor-87-Odd}.
\end{example}

\begin{example}[Witt-Bose-Shrikhande spaces]\label{ex:witt}
	This space is a $2$-design with parameters  $(2^{a-1}(2^{a}-1), 2^{a-1},1)$ which can be defined from the group $\PSL_2(q)$ with $q=2^{a}$ for $a\geq 3$ \cite{a:BDD-1988}. In this incidence structure which is denoted by $\W(q)$, the points are the dihedral subgroups of $\PSL_{2}(q)$ of order $2(q+1)$, the blocks are the involutions of $\PSL_2(q)$, and a point is incident with a block precisely when the dihedral subgroup contains the involution. An almost simple group $G$ with socle $X=\PSL_2(q)$ acts flag-transitively on Witt-Bose-Shrikhande space. Moreover, this space is not a symmetric design.
\end{example}

\begin{example}[Hermitian unitals]\label{ex:hermition}
	The Hermitian unital with parameters  $(q^{3}+1,q+1,1)$ is a well-known example of flag-transitive $2$-designs \cite{a:Kantor-85-Homgeneous}. Let $V$ be a three-dimensional vector space over the field $\Fbb_{q^2}$ with a non-degenerate Hermitian form. The Hermitian unital is an incidence structure whose points are $q^3+1$ totally isotropic $1$-spaces in $V$, the blocks are the sets of $q-1$ points lying in a non-degenerate $2$-space, and the incidence is given by inclusion. This structure is not symmetric and any group $G$ with $\PSU_3(q)\leq G\leq \PGU_3(q)$ acts flag-transitively on Hermitian unital design.
\end{example}

\begin{example}[Ree unitals]\label{ex:ree-unital}
	The Ree Unital spaces $\U_{R}(q)$ are first discovered by L\"{u}neburg \cite{a:Luneburg-66}, and these examples arose from studying flag-transitive linear spaces \cite{a:Kantor-85-Homgeneous,a:Kleidman-Exp}. This design has parameters  $(q^{3}+1,q+1,1)$ with $q=3^a\geq 27$.   The points and blocks of $\U_{R}(q)$ are the Sylow $3$-subgroups and the involutions of $^{2}\!\G_{2}(q)$, respectively, and a point is incident with a block if the  block normalizes the  point. This incidence structure is a linear space and any group with $^{2}\!\G_{2}(q)\leq G\leq \Aut(^{2}\!\G_{2}(q))$ acts flag-transitively. This design is not symmetric. Note for $q=3$ that the Ree Unital $\U_{R}(3)$ is isomorphic to the  Witt-Bose-Shrikhande space $\W(8)$ as $^{2}\!\G_{2}(3)'$ is isomorphic to $\PSL_{2}(8)$.
\end{example}

\begin{example}[Ree designs]\label{ex:ree-designs}
	Suppose that $G$ is an almost simple group with socle $X={}^2\!\G_{2}(q)$ for $q=2^a$ and $a\geq 3$ odd. Let $H$, $K_{1}$ and $K_{2}$ be subgroups of $G$ such that $H\cap X\cong q^{3}{:}(q-1)$, $K_{1}\cap X=q{:}(q-1)$ and $K_{2}\cap X \cong q^{2}{:}(q-1)$. The coset geometries $(X,H\cap X,K_{i}\cap X)$ may give rise to the $2$-designs with parameters  $v=q^3+1$, $b=q^{3-i}(q^{3}+1)$, $r=q^{3}$,  $k=q^{i}$ and $\lambda=q^{i}-1$, for $i=1,2$. Since $G$ is $2$-transitive on the points set of this structure and $\gcd(r,\lambda) = 1$,  $X$ is flag-transitive \cite[2.3.8]{b:Dembowski}. Note that $H\cap K_{i} \cap X$ is a cyclic group of order $q-1$. Let $B_{i}$ be an orbit of $K_{i}\cap X$ of length $k=q^{i}$ with $i=1,2$. If $\Pmc=\{1,\ldots,v\}$, then since $X$ is $2$-transitive, \cite[Proposition~4.6]{b:Beth-I} gives rise to a $2$-design $\Dmc_{i}=(\Pmc,B_{i}^{X})$ with parameters $(q^3+1,q^{i},q^{i}-1)$, for $i=1,2$, which is not symmetric, and the group $G$ is flag-transitive on $\Dmc_i$. For $q=27$,  in \cite[Table 1]{a:A-Exp-CP}, we introduced base blocks for these type of designs, and explicit constructions for these designs are given in \cite{a:D-Ree}.
\end{example}

\begin{example}[Suzuki designs]\label{ex:suzuki}
	Suppose that $G$ is an almost simple group with socle $X={}^2\!\B_{2}(q)$ for $q=2^a$ and $a\geq 3$ odd. Let $H$ and  $K$ be subgroups of $G$ such that $H\cap X\cong q^{2}{:}(q-1)$ and $K\cap X \cong q{:}(q-1)$. The coset geometry $(X,H\cap X,K\cap X)$ may give rise to a $2$-design with parameters  $v=q^2+1$, $b=q(q^{2}+1)$, $r=q^{2}$,  $k=q$ and $\lambda=q-1$,  see \cite{a:A-Exp-CP}. By \cite[2.3.8]{b:Dembowski}, $X$ is flag-transitive. If $\Pmc=\{1,\ldots,v\}$ and  $B$ is an orbit of $K\cap X$ of length $k=q$, then by \cite[Proposition~4.6]{b:Beth-I}, $(\Pmc,B^{X})$ is a $2$-design  with parameters $(q^2+1,q,q-1)$ which is not symmetric.
	For $q\in \{8,32\}$, we construct these type of designs with explicit base blocks in \cite[Table~1]{a:A-Exp-CP}, and in \cite{a:A-Sz} explicit constructions of these designs are given. There is another subtle construction for these designs in geometry, by taking points as ovoids in $\PG_{3}(q)$ and blocks as pointed conics minus the distinguished points.    
\end{example}

\begin{example}\label{ex:other}
	The design $\Dmc=(\Pmc, \Bmc)$ with parameters $(v,k,\lambda)$ listed in Table~\ref{tbl:main} is the unique design with flag-transitive automorphism
	group $G$ as in the seventh column of Table~\ref{tbl:main}. The base block, point-stabiliser and block-stabiliser of $\Dmc$ are also given in the same table with appropriate references in each case. The base blocks for the designs in lines $2$, $6$ and $12$ are $\{1,2,3\}$, $\{1,2,3,5,11\}$ and $\{1, 2, 3, 8, 9, 10, 11\}$, respectively.
	We note here that the designs in line $2$ is the well-known projective plane, namely Fano plane, and the design in line $12$ is viewed as a projective space, see Example~\ref{ex:proj-space}.
\end{example}
\begin{table}
	\caption{ Some nontrivial $2$-design $\Dmc$ with $\gcd(r,\lambda)=1$ admitting flag-transitive and point-primitive automorphism group $G$.}\label{tbl:main}
	\resizebox{\textwidth}{!}{
		\begin{tabular}{clllllllll}
			\noalign{\smallskip}\hline\noalign{\smallskip}
			Line &
			$v$ &
			$b$ &
			$r$ &
			$k$ &
			$\lambda$ &
			$G$ &
			$G_{\alpha}$&
			$G_{B}$ &
			$\Aut(\Dmc)$
			\\
			\noalign{\smallskip}\hline\noalign{\smallskip}
			\multirow{2}{*}{$1$} &
			$6$ &
			$10$ &
			$5$ &
			$3$ &
			$2$ &
			$\PSL_{2}(5)$ &
			$\D_{10}$ &
			$\S_{3}$ &
			$\PSL_{2}(5)$
			\\
			& \multicolumn{7}{l}{Base block: \{1,2,3\}} &
			\multicolumn{2}{l}{References: \cite{b:Handbook,a:Zhou-PSL2-non-sym,a:Zhou-CP-nonsym-alt}} \\
			\noalign{\smallskip}\hline\noalign{\smallskip}
			\multirow{2}{*}{$2$} &
			$7$ &
			$7$ &
			$3$ &
			$3$ &
			$1$ &
			$\PSL_{2}(7)$ &
			$\S_{4}$ &
			$\S_{4}$ &
			$\PSL_{2}(7)$
			\\
			& \multicolumn{7}{l}{Base block: \{1,2,3\}, $\PG_{2}(2)$} &
			\multicolumn{2}{l}{References: \cite{a:ABD-PSL2,b:Handbook,a:Kantor-85-2-trans,a:Regueiro-reduction}}\\
			\noalign{\smallskip}\hline\noalign{\smallskip}
			\multirow{2}{*}{$3$} &
			$8$ &
			$14$ &
			$7$ &
			$4$ &
			$3$ &
			$\PSL_{2}(7)$&
			$7{:}3$ &
			$\A_{4}$ &
			$2^{3}{:}\PSL_{2}(7)$
			\\
			& \multicolumn{7}{l}{Base block: \{1,2,3,5\}}
			&
			\multicolumn{2}{l}{References: \cite{a:ABD-Un-CP}}\\
			\noalign{\smallskip}\hline\noalign{\smallskip}
			\multirow{2}{*}{$4$} &
			$28$ &
			$36$ &
			$9$ &
			$7$ &
			$2$ &
			$\PSL_{2}(8)$ &
			$\D_{18}$ &
			$\D_{14}$ &
			$\PSL_{2}(8){:}3$
			\\
			& \multicolumn{7}{l}{Base block: \{1,6,12,13,14,24,28\}}&
			\multicolumn{2}{l}{References: \cite{b:Handbook,a:Zhou-PSL2-non-sym}} \\
			\noalign{\smallskip}\hline\noalign{\smallskip}
			\multirow{2}{*}{$5$} &
			$10$ &
			$15$ &
			$9$ &
			$6$ &
			$5$ &
			$\PSL_{2}(9)$&
			$3^2{:}4$&
			$\S_{4}$ &
			$\S_{6}$
			\\
			& \multicolumn{7}{l}{Base block: \{1,2,3,4,5,6\}} &
			\multicolumn{2}{l}{References: \cite{a:ABD-PSL2,b:Handbook,a:Zhou-PSL2-non-sym,a:Zhou-CP-nonsym-alt}} \\
			\noalign{\smallskip}\hline\noalign{\smallskip}
			\multirow{2}{*}{$6$} &
			$11$ &
			$11$ &
			$5$ &
			$5$ &
			$2$ &
			$\PSL_{2}(11)$ &
			$\PSL_{2}(5)$ &
			$\PSL_{2}(5)$ &
			$\PSL_{2}(11)$
			\\
			& \multicolumn{7}{l}{Base block: \{1,2,3,5,11\}, a Hadamard design} &
			\multicolumn{2}{l}{References: \cite{a:ABD-PSL2,b:Handbook,a:Kantor-85-2-trans,a:Regueiro-reduction}} \\
			\noalign{\smallskip}\hline\noalign{\smallskip}
			\multirow{2}{*}{$7$} &
			$12$ &
			$22$ &
			$11$ &
			$6$ &
			$5$ &
			$\M_{11}$ &
			$\PSL_{2}(11)$ &
			$\A_{6}$ &
			$\M_{11}$
			\\
			& \multicolumn{7}{l}{Base block: \{ 1, 2, 3, 4, 11, 12\}} &
			\multicolumn{2}{l}{References: \cite{a:Zhan-CP-nonsym-sprodic}}\\
			\noalign{\smallskip}\hline\noalign{\smallskip}
			\multirow{2}{*}{$8$} &
			$22$ &
			$77$ &
			$21$ &
			$6$ &
			$5$ &
			$\M_{22}$ &
			$\PSL_{3}(4)$ &
			$2^{4}{:}\A_{6}$ &
			$\M_{22}$
			\\
			& \multicolumn{7}{l}{Base block: \{  1, 2, 3, 4, 7, 18 \}} &
			\multicolumn{2}{l}{References: \cite{a:Zhan-CP-nonsym-sprodic}}\\
			\noalign{\smallskip}\hline\noalign{\smallskip}
			\multirow{2}{*}{$9$} &
			$22$ &
			$77$ &
			$21$ &
			$6$ &
			$5$ &
			$\M_{22}{:}2$ &
			$\PSL_{3}(4){:}2$ &
			$2^{4}{:}\S_{6}$ &
			$\M_{22}{:}2$
			\\
			& \multicolumn{7}{l}{Base block: \{ 1, 2, 3, 4, 7, 18 \}} &
			\multicolumn{2}{l}{References: \cite{a:Zhan-CP-nonsym-sprodic}}\\
			\noalign{\smallskip}\hline\noalign{\smallskip}
			\multirow{2}{*}{$10$} &
			$10$ &
			$15$ &
			$9$ &
			$6$ &
			$5$ &
			$\S_{6}$&
			$\S_{3}^{2}{:}2$&
			$\S_{4}{\times}2$ &
			$\S_6$
			\\
			& \multicolumn{7}{l}{Base block: \{1,2,3,4,5,6\}} &
			\multicolumn{2}{l}{References: \cite{a:ABD-PSL2,b:Handbook,a:Zhou-PSL2-non-sym,a:Zhou-CP-nonsym-alt}}\\
			\noalign{\smallskip}\hline\noalign{\smallskip}
			\multirow{2}{*}{$11$} &
			$15$ &
			$35$ &
			$7$ &
			$3$ &
			$1$ &
			$\A_7$  &
			$\PSL_2(7)$ &
			$(3\times \A_{4}){:}2$ &
			$\A_7$
			\\
			& \multicolumn{7}{l}{Base block: \{1, 4, 5\}} &
			\multicolumn{2}{l}{References: \cite{b:Handbook,a:Kantor-85-2-trans,a:Zhou-CP-nonsym-alt}}\\
			\noalign{\smallskip}\hline\noalign{\smallskip}
			\multirow{2}{*}{$12$} &
			$15$ &
			$15$ &
			$7$ &
			$7$ &
			$3$ &
			$\A_7$ &
			$\PSL_2(7)$ &
			$\PSL_2(7)$ &
			$\A_7$
			\\
			& \multicolumn{7}{l}{Base block: \{1, 2, 3, 8, 9, 10, 11\},  $\PG_{3}(2)$}&
			\multicolumn{2}{l}{References: \cite{b:Handbook,a:Zhuo-CP-sym-alt}}\\
			\noalign{\smallskip}\hline\noalign{\smallskip}
			\multirow{2}{*}{$13$} &
			$15$ &
			$35$ &
			$7$ &
			$3$ &
			$1$ &
			$\A_{8}$ &
			$2^3{:}\PSL_3(2)$&
			$\A_{4}^{2}{:}2{:}2$ &
			$\A_{8}$
			\\
			& \multicolumn{7}{l}{Base block: \{1, 4, 5\}} &
			\multicolumn{2}{l}{References: \cite{b:Handbook,a:Zhou-CP-nonsym-alt}}\\
			\noalign{\smallskip}\hline\noalign{\smallskip}
			Comments & \multicolumn{9}{p{10cm}}{The subgroups $G_{\alpha}$ and $G_{B}$ are point-stabiliser and block-stabiliser, respectively.  }\\
			& \multicolumn{9}{p{10cm}}{ The first row in each line consists of $v$, $b$, $r$, $k$, $\lambda$, $G$, $G_{\alpha}$, $G_{B}$, $\Aut(\Dmc)$   }\\
			& \multicolumn{9}{p{10cm}}{ The second row in each line consists of a base block for $\Dmc$, the name of the design if it is known and appropriate references.}\\
			
		\end{tabular}
	}
\end{table}

\section{Preliminaries}\label{sec:pre}

In this section, we state some useful facts in both design theory and group theory. Lemma \ref{lem:New} below is an elementary result on subgroups of almost simple groups.

\begin{lemma}\label{lem:New}{\rm \cite[Lemma 2.2]{a:ABD-PSL2}}
	Let $G$ be an almost simple group with socle $X$, and let $H$ be maximal in $G$ not containing $X$. Then $G=HX$ and $|H|$ divides $|\Out(X)|{\cdot}|H\cap X|$.
\end{lemma}

\begin{lemma}\label{lem:Tits}{\rm (Tits' Lemma~\cite[1.6]{a:Seitz-TitsLemma})}
	If $X$ is a simple group of Lie type in characteristic $p$, then any proper subgroup of index prime to $p$ is contained in a parabolic subgroup
	of $X$.
\end{lemma}

If a group $G$ acts on a set $\Pmc$ and $\alpha\in \Pmc$, the \emph{subdegrees} of $G$ are the size of orbits of the action of the point-stabiliser $G_\alpha$ on $\Pmc$.

\begin{lemma}\label{lem:subdeg}{\rm \cite[3.9]{a:LSS1987}}
	If $X$ is a group of Lie type in characteristic $p$, acting on the set of cosets of a maximal parabolic subgroup, and $X$ is neither $\PSL_n(q)$, $\POm_{n}^{+}(q)$
	(with $n/2$ odd), nor $\E_{6}(q)$, then there is a unique subdegree which is a power of $p$.
\end{lemma}

\begin{remark}\label{rem:subdeg}
	We remark that even in the cases excluded in Lemma~\ref{lem:subdeg}, many of the maximal parabolic subgroups still have the property as asserted, see proof of \cite[Lemma 2.6]{a:Saxl2002}.
\end{remark}

\begin{lemma}{\rm \cite{a:Liebeck-Odd}}\label{lem:odd-deg}
	If $X$ is a simple group of Lie type in odd characteristic, and $X$ is neither $\PSL_n(q)$, nor $E_6(q)$, then the index of any parabolic subgroup is even.
\end{lemma}

\begin{lemma}{\rm \cite[Propositions 1 and 2]{a:LPS2} and \cite[Sections 3-7]{a:Saxl2002} } \label{lem:subdeg-PSp}
	Let $G$ be an almost simple group with socle $X$  being a finite classical simple group of Lie type with dimension at least three. Suppose that $H$ is a maximal subgroup of $G$ not containing $X$. Then the action of $G$ on the cosets of $H$ has subdegrees dividing the numbers $d$ listed in the fourth column of {\rm Table~\ref{tbl:subdeg-PSp}}.
\end{lemma}

\begin{table}
	\caption{Some subdegrees of finite simple classical groups $X$ on its right coset action of $H$.}\label{tbl:subdeg-PSp}
	\scriptsize
	\resizebox{\textwidth}{!}{
		\begin{tabular}{cllll}
			\noalign{\smallskip}\hline\noalign{\smallskip}
			Class & $X$ & Type of $H$ & $d$  &  Conditions  \\
			\noalign{\smallskip}\hline\noalign{\smallskip}
			$\Cmc_1$ &
			$\PSL_{n}(q)$ &
			$P_i$ &
			$\dfrac{q(q^i-1)(q^{n-i}-1)}{(q-1)^2}$ &
			$1<i\leq n/2$, $n>2$
			\\
			$\Cmc_1$ &
			$\PSL_{n}(q)$ &
			$P_i$ &
			$\dfrac{q^9(q^4-1)}{(q-1)}$ &
			$i=3$, $n=7$
			\\
			$\Cmc_1$ &
			$\PSL_{n}(q)$ &
			$P_i$ &
			$\dfrac{q(q+1)(q^{n-2}-1)}{(q-1)}$ &
			$i=2$, $n\geq 4$
			\\
			&
			&
			$P_i$&
			$\dfrac{q^4(q^{n-2}-1)(q^{n-3}-1)}{(q^2-1)(q-1)}$ &
			$i=2$, $n\geq 4$
			\\
			$\Cmc_1$ &
			$\PSL_{n}(q)$ &
			$\GL_i(q){\bigoplus}\GL_{n-i}(q)$ &
			$\dfrac{q^{n-2}(q^{n-1}-1)}{(q-1)}$ &
			$i=1$, $n>2$
			\\
			$\Cmc_2$ &
			$\PSL_{n}(q)$ &
			$\GL_{m}(q)\wr\S_t$ &
			$2n(n-1)(q-1)$ &
			$n=mt>2$, $m=1$
			\\
			$\Cmc_2$ &
			$\PSL_{n}(q)$ &
			$\GL_{m}(q)\wr\S_t$ &
			$\dfrac{t(t-1)(q^m-1)^2}{(q-1)}$ &
			$n=mt>2$, $m>1$
			\\
			$\Cmc_3$ &
			$\PSL_{n}(q)$ &
			$\GL_{m}(q^t)$ &
			$(q^n-1)(q^{n-2}-1)$ &
			$t=2$, $n=2m\geq8$
			\\
			$\Cmc_5$ &
			$\PSL_{n}(q)$ &
			$\GL_{n}(q_0)$ &
			$(q_0^n-1)(q_0^{n-1}-1)$ &
			$q=q_0^2$, $n\geq5$
			\\
			$\Cmc_8$ &
			$\PSL_{n}(q)$ &
			$\Sp_n(q)$ &
			$(q^n-1)(q^{n-2}-1)$ &
			$n$ even, $n\geq 8$
			\\
			$\Cmc_8$ &
			$\PSL_{n}(q)$ &
			$\GU_n(q_0)$ &
			$(q_0^n-(-1)^n)$ &
			$q=q_0^2$, $n\geq4$
			\\
			&
			&
			&
			${\cdot}(q_0^{n-1}-(-1)^{n-1})$ & \\
			
			$\Cmc_2$ &
			$\PSp_{2m}(q)$ &
			$\Sp_{2i}(q) \wr \S_t$ &
			$\dfrac{t(t-1)(q^{2i}-1)^2}{2(q-1)}$ &
			$m=it$, $m\geq2$
			\\
			$\Cmc_2$ &
			$\PSp_4(q)$ &
			$\Sp_{2}(q) \wr \S_2$ &
			$(q^2-1)(q+1)$ &
			$q$ odd
			\\
			&
			&
			$\Sp_{2}(q) \wr \S_2$ &
			$q(q^2-1)/2$ &
			$q$ odd
			\\
			&
			&
			$\Sp_{2}(q) \wr \S_2$ &
			$q(q^2-1)(q-3)/2$ &
			$q$ odd
			\\
			&
			&
			$\Sp_{2}(q) \wr \S_2$ &
			$(q^2-1)(q+1)$ &
			$q$ even
			\\
			&
			&
			$\Sp_{2}(q) \wr \S_2$ &
			$ q(q^2-1)(q-2)/2$ &
			$q$ even
			\\
			$\Cmc_3$ &
			$\PSp_{2m}(q)$ &
			$\Sp_{2i}(q^t)$ &
			$a(q^{4i}-1)$ &
			$m=it\geq2$, $t=2$
			\\
			$\Cmc_3$ &
			$\PSp_4(q)$ &
			$\Sp_2(q^2)$ &
			$(q^2+1)(q-1)$&
			$q$ odd
			\\
			&
			&
			$\Sp_2(q^2)$ &
			$q(q^2+1)/2$&
			$q$ odd
			\\
			&
			&
			$\Sp_2(q^2)$ &
			$q(q^2+1)(q-3)/2$ &
			$q$ odd
			\\
			&
			&
			$\Sp_2(q^2)$ &
			$(q^2+1)(q-1)$ &
			$q$ even
			\\
			&
			&
			$\Sp_2(q^2)$ &
			$q(q^2+1)(q-2)/2$ &
			$q$ even
			\\
			$\Cmc_8$ &
			$\PSp_{2m}(q)$ &
			$\GO_{2m}^{\epsilon}(q)$ &
			$(q^m - \epsilon)(q^{m-1}+ \epsilon)$ &
			$q$ even, $m\geqslant 3$, $\e=\pm$
			\\
			&
			&
			$\GO_{2m}^{\epsilon}(q)$ &
			$(q^{m-1})(q^m - \epsilon)(q-2)/2$ &
			$q$ even, $m\geqslant 3$, $\e=\pm$
			\\
			$\Cmc_1$ &
			$\POm_{2m}^{+}(q)$ &
			$P_m$ &
			$\dfrac{q(q^2+1)(q^5-1)}{(q-1)}$ &
			$m=5$
			\\
			&
			&
			$P_m$ &
			$\dfrac{q^6(q^5-1)}{(q-1)}$ &
			$m=5$ \\
			$\Cmc_1$ &
			$\POm_{2m+1}(q)$ &
			$N_1^\delta$ &
			$(q^m - \delta)(q^{m} + \delta)$ &
			$m\geq3$, $\delta=\pm$, $q$ odd
			\\
			&
			&
			$N_1^\delta$ &
			$q^{m-1}(q^{m} - \delta)/2$ &
			$m\geq3$, $\delta=\pm$, $q$ odd
			\\
			&
			&
			$N_1^\delta$ &
			$q^{m-1}(q^m - \delta)(q-3)/2$ &
			$m\geq3$, $\delta=\pm$, $q$ odd
			\\
			$\Cmc_1$ &
			$\POm_{2m}^{\e}(q)$ &
			$N_1$ &
			$q^{2m-2}-1$ &
			$m\geq4$, $q \equiv 1 \pmod 4$
			\\
			&
			&
			$N_1$ &
			$q^{m-1}(q^{m-1}+\e)/2$ &
			$m\geq4$, $q \equiv 1 \pmod 4$
			\\
			&
			&
			$N_1$ &
			${q^{m-1}(q^{m-1}-\e)(q-1)/4}$ &
			$m\geq4$, $q \equiv 1 \pmod 4$
			\\
			&
			&
			$N_1$ &
			${q^{m-1}(q^{m-1}+\e)(q-5)/4}$ &
			$m\geq4$, $q \equiv 1 \pmod 4$
			\\
			&
			&
			$N_1$ &
			$q^{2m-2}-1$ &
			$m\geq4$, $q \equiv 3 \pmod 4$
			\\
			&
			&
			$N_1$ &
			$q^{m-1}(q^{m-1}-\e)/2$ &
			$m\geq4$, $q \equiv 3 \pmod 4$
			\\
			&
			&
			$N_1$ &
			${q^{m-1}(q^{m-1}-\e)(q-3)/4}$ &
			$m\geq4$, $q \equiv 3 \pmod 4$
			\\
			&
			&
			$N_1$ &
			${q^{m-1}(q^{m-1}+\e)(q-3)/4}$ &
			$m\geq4$, $q \equiv 3 \pmod 4$
			\\
			&
			&
			$N_1$ &
			$q^{2m-2}-1$ &
			$m\geq4$, $q$ even
			\\
			&
			&
			$N_1$ &
			$q^{m}(q^{m-1}-\e)/2$ &
			$m\geq4$, $q$ even
			\\
			&
			&
			$N_1$ &
			${q^{m-1}(q^{m-1}+\e)(q-2)/2}$ &
			$m\geq4$, $q$ even
			\\
			$\Cmc_1$ &
			$\POm_{2m}^{\e}(q)$ &
			$N_2^\delta$ &
			$a_{2'}(q-\delta)(q^{m-1}-\e \delta)$ &
			$m\geq4$, $q$ even
			\\
			$\Cmc_1$ &
			$\POm_{2m}^{\e}(q)$ &
			$N_i^\delta$ &
			$a_{2'}(q^{i/2}-\delta)(q^{(i-2)/2}+\delta)$ &
			$m\geq4$, $q$ even,
			\\
			&
			&
			&
			${\cdot}(q^{(n-i)/2}+\delta')$ &
			$2<i\leq m$, $i$ even,
			\\
			&
			&
			&
			${\cdot}(q^{(n-i-2)/2}+\delta')$ &
			$\delta'=-\e \delta$,
			\\
			$\Cmc_2$ &
			$\POm_{2m}^{+}(q)$ &
			$\GL_m(q){\cdot}2$ &
			$a(q^m-1)(q^{m-1}-1)$ &
			$m \geq5$
			\\
			$\Cmc_3$ &
			$\POm_{2m}^{\e}(q)$ &
			$\GU_{m}(q)$ &
			$a(q^m-(-1)^m)(q^{m-1}-(-1)^{m-1})$ &
			$m\geq5$, $\e=(-1)^m$,
			\\
			&
			&
			&
			${\cdot}(q^{m-1}-(-1)^{m-1})$ &
			$q$ even
			\\
			$\Cmc_{5}$ &
			$\POm^{+}_{2m}(q)$ &
			$\GO_{2m}^{\e'}(q_0)$  &
			$ac\gcd(4,q^{m}-1)\cdot (q_0^{m}-\e')$&
			$m\geq4$, $\e'=\pm$, $q=q_{0}^2$,
			\\
			&
			&
			&
			${\cdot}(q_0^{m-1}+\e')$&
			$(m,c)=(4,6)$ or $c=2$
			\\
			\noalign{\smallskip}\hline\noalign{\smallskip}
		\end{tabular}
	}
\end{table}

\begin{lemma}\label{lem:six} {\rm \cite[Lemmas 5 and 6]{a:Zhou-CP-nonsym-alt}}
	Let $\Dmc$ be a $2$-design with $r$ is coprime to $\lambda$, and let $G$ be a flag-transitive automorphism group of $\Dmc$. If $\alpha$ is a point in $\Pmc$ and $H=G_{\alpha}$, then
	\begin{enumerate}[\rm \quad (a)]
		\item $r(k-1)=\lambda(v-1)$. In particular, if $\gcd(r,\lambda)=1$, then $r$ divides $v-1$ and $\gcd(r,v)=1$;
		\item $vr=bk$;
		\item $r$ divides $|H|$, and $\lambda v<r^2$;
		\item $r$ divides all nontrivial subdegrees $d$ of $G$.
	\end{enumerate}
\end{lemma}

\begin{lemma}\label{lem:embed} {\rm \cite[2.2.5]{b:Dembowski} and \cite[Theorem II 6.27]{b:Handbook}}
	Let $\Dmc$ be a $2$-design with parameters $(v,k,\lambda)$ design with $\lambda\leq 2$. If $\Dmc$ satisfies $r = k + \lambda$, then $\Dmc$ is
	embeddable in a symmetric $(v + k + \lambda,k + \lambda,\lambda)$ design.
\end{lemma}

For a given positive integer $n$ and a prime divisor $p$ of $n$, we denote the $p$-part of $n$ by $n_{p}$, that is to say, $n_{p}=p^{t}$ with $p^{t}\mid n$ but $p^{t+1}\nmid n$.

\begin{corollary}{\rm \cite[Corollary 2.1]{a:ABCD-PrimeRep}}\label{cor:large}
	Let $\Dmc$ be a flag-transitive $2$-design with automorphism group $G$. Then $|G|\leq |G_{\alpha}|^3$, where $\alpha$ is a point in $\Dmc$.
\end{corollary}


\begin{corollary}\label{cor:large-2}
	Suppose that $\Dmc$ is a $2$-design with $\gcd(r,\lambda)=1$ admitting a flag-transitive, point-primitive almost simple automorphism group $G$ with simple socle $X$ of Lie type in characteristic $p$, and the stabilizer $H=G_\alpha$ is not a parabolic subgroup of $G$. Then
	$|G|<|H| \cdot |H|_{p'}^2$, and hence $|X|<|\Out(X)|_{p'}^2 \cdot |H \cap X| \cdot |H \cap X|_{p'}^2$.
\end{corollary}

\begin{proof}
	We know by Lemma~\ref{lem:Tits} that $p$ divides  $v=|G:H|$, and so
	$\gcd(p,v-1)= 1$. Lemma~\ref{lem:six}(a) implies that $r$ divides $v-1$. Thus $\gcd(r,p)=1$, and since $r$ divides $|H|$,  $r \leq  |H|_{p'}$. Therefore, $v<r^2$ implies that $|G:H|<|H|_{p'}^2$, or equivalently, $|G|<|H|{\cdot}|H|_{p'}^2$.  Moreover, since $|G:H|=|X:H\cap X|$ and $|H|\leq |\Out(X)| {\cdot} |H\cap X|$, the inequality $|G:H|<|H|_{p'}^2$ yields $|X|<|\Out(X)|_{p'}^2 {\cdot} |H \cap X| {\cdot} |H \cap X|_{p'}^2$.
\end{proof}

\begin{proposition}\label{prop:flag}
	Let $\Dmc$ be a $2$-design with $\gcd (r, \lambda)=1$ admitting a flag-transitive automorphism group $G$. Then $G$ is point-primitive of almost simple or affine type.
\end{proposition}
\begin{proof}
	Since $r$ is coprime to $\lambda$, it follows from \cite[2.3.7(a)]{b:Dembowski}  that $G$ is point-primitive. We now apply
	\cite[Theorem]{a:Zieschang-88} and conclude that $G$ is of almost simple or affine type.
\end{proof}

\begin{lemma}{\rm \cite[Corollary 2.5]{a:ABCD-PrimeRep}}\label{lem:divisible}
	Suppose that $\Dmc$ is a $2$-design with $\gcd(r,\lambda)=1$. Let $G$ be a flag-transitive automorphism group of $\Dmc$ with simple socle $X$ of Lie type in characteristic $p$. If the point-stabiliser $H=G_{\alpha}$ contains a normal quasi-simple subgroup $K$ of Lie type in characteristic $p$ and $p$ does not divide $|Z(K)|$, then either $p$ divides $r$, or $K_{B}$ is contained in a parabolic subgroup $P$ of $K$ and $r$ is divisible by $|K:P|$.
\end{lemma}

The maximal subgroups of classical groups have been determined in  Aschbacher's Theorem~\cite{a:Aschbacher-84} which says that such a maximal subgroup $H$ lies in one of the eight geometric families $\Cmc_{i}$ of subgroups of $G$, or it is in the family $\Smc$ of almost simple subgroups with some irreducibility conditions. We follow the description of these subgroups as in \cite{b:KL-90}. A rough description of the $\Cmc_i$ families is given in Table \ref{tbl:max}. In what follows, if $H$ belongs to the family $\Cmc_{i}$, for some $i$, then we sometimes say that $H$ is a $\Cmc_{i}$-subgroup. We also denote by $\,^{\hat{}}H$ the pre-image of the group $H$ in the corresponding linear group.

\begin{table}
	\caption{The geometric subgroup collections.}\label{tbl:max}
	\resizebox{\textwidth}{!}{
		\begin{tabular}{clllll}
			\noalign{\smallskip}\hline\noalign{\smallskip}
			Class & Rough description\\
			\noalign{\smallskip}\hline\noalign{\smallskip}
			$\Cmc_1$ & Stabilisers of subspaces of $V$\\
			$\Cmc_2$ & Stabilisers of decompositions $V=\bigoplus_{i=1}^{t}V_i$, where $\dim V_i = a$\\
			$\Cmc_3$ & Stabilisers of prime index extension fields of $\mathbb{F}$\\
			$\Cmc_4$ & Stabilisers of decompositions $V=V_1 \otimes V_2$\\
			$\Cmc_5$ & Stabilisers of prime index subfields of $\mathbb{F}$\\
			$\Cmc_6$ & Normalisers of symplectic-type $r$-groups in absolutely irreducible representations\\
			$\Cmc_7$ & Stabilisers of decompositions $V=\bigotimes_{i=1}^{t}V_i$, where $\dim V_i = a$\\
			$\Cmc_8$ & Stabilisers of non-degenerate forms on $V$\\ \noalign{\smallskip}\hline\noalign{\smallskip}
		\end{tabular}
	}
\end{table}

\begin{lemma}\label{lem:coprime}
	Suppose that $G$ is an almost simple group whose socle $X$  is a finite simple classical group with dimension at least three. Suppose also that $H$ is a maximal geometric subgroup of $G$ with $H\not \in \Cmc_{5}\cup \Cmc_6$ and $H$ is not a $\Cmc_2$-subgroup of type $\GO_1(p) \wr \S_n$ when $X$ is an orthogonal group. If $|H\cap X|_p<|\Out(X)|$, then $X$ is $\PSL_{n}(q)$, $\PSU_{n}(q)$ or $\POm^\e_n(q)$, and
	\begin{enumerate}[\rm(a)]
		\item if $X=\PSL_n(q)$ with $n\geq3$, then
		\begin{enumerate}[\rm(i)]
			\item $H$ is a $\Cmc_1$-subgroup of type $\GL_1(q)\oplus \GL_{2}(q)$ with $q=4,16$;
			\item $H$ is a $\Cmc_2$-subgroup of type $\GL_1(q) \wr \S_n$;
			\item $H$ is a $\Cmc_3$-subgroup of type $\GL_1(q^n)$;
			\item $H$ is a $\Cmc_8$-subgroup of type $\GO_{3}(q)$ with $q$ odd;
			\item $H$ is a $\Cmc_8$-subgroup of type $\GO_{4}^{\e}(q)$ with $q$ odd and $\e=\pm$;
			\item $H$ is a $\Cmc_8$-subgroup of type $\GU_{3}(2)$;
		\end{enumerate}
		\item if $X=\PSU_n(q)$ with $n\geq 3$, then
		\begin{enumerate}[\rm (i)]
			\item $H$ is a $\Cmc_1$-subgroup of type $\GU_1(q)\perp \GU_{2}(q)$ with $q=2,5,8$;
			\item $H$ is a $\Cmc_2$-subgroup of type $\GU_1(q)\wr\S_n$;
			\item $H$ is a $\Cmc_2$-subgroup of type $\GL_2(9)$;
			\item $H$ is a $\Cmc_3$-subgroup of type $\GU_1(q^n)$;
		\end{enumerate}
		\item if $X=\POm^\e_n(q)$ with $\e \in\{\circ ,+,-\}$ and $n\geq8$ if $n$ is even and $n\geq7$ if $n$ is odd, then $H$ is a $\Cmc_2$-subgroup of type $\GO^{\e'}_2(q) \wr\S_{n/2}$ with $\e=(\e')^{n/2}$.
	\end{enumerate}
\end{lemma}

\begin{proof}
	Part (b) follows from \cite[Lemma~3.14]{a:ABD-Un-CP}. Here, we prove part (a) and the proof for the remaining finite simple classical groups is similar. In what follows, we use  the same approach as in \cite[Lemma~3.14]{a:ABD-Un-CP}. Note in conclusion that for $X=\PSp_{n}(q)$ with $n\geq4$, we always have $|H\cap X|_p\geq |\Out(X)|$.
	
	Suppose that $X=\PSL_n(q)$ with $n\geq3$ and $q=p^a$. Since $H$ is a maximal geometric subgroup in $G$, then by Aschbacher's Theorem~\cite{a:Aschbacher-84}, the subgroup $H$ lies in one of the families $\Cmc_i$ for some $i=1, \ldots,8$. Let $H\not \in \Cmc_6$. We will analyse each of these
	cases separately.\smallskip
	
	\noindent \textbf{(1)} If $H\in \Cmc_1$, then $H$ is reducible, and $H$ stabilises a subspace of $V$ of dimension $i$ with $1\leq  i\leq n/2$ or $G$ contains a graph automorphism and $H$ stabilises a pair $\{U,W\}$ of subspaces of dimension $i$ and $n-i$ with $i<n/2$.
	
	Suppose first that $H\cong P_{i}$ for some $1\leq  i\leq n/2$. Then by~\cite[Proposition 4.1.17]{b:KL-90}, $|H\cap X|_p=q^{n(n-1)/2}$. Since $|\Out(X)|=2a{\cdot}\gcd(n,q-1)$, the inequality $|H\cap X|_p<|\Out(X)|$ implies that $q^{n(n-1)/2}<2a{\cdot}\gcd(n,q-1)<2aq$. Note that $n\geq3$ and $\gcd(n,q-1)<q$.
	Thus $q^{n(n-1)-2}<4a^2$, and since  $q\geq 2a$, we have that $q^2\leq q^{n(n-1)-4}<1$, which is impossible.
	
	Suppose now that $H=N_{G}(U,W)$ with $\dim(U)=i$. If $U\subset W$, then by~\cite[Proposition 4.1.22]{b:KL-90}, we also have that $|H\cap X|_p=q^{n(n-1)/2}$, and so by the same argument as above, this case cannot occur. If $U\cap W=0$, then by~\cite[Proposition 4.1.4]{b:KL-90}, $|H\cap X|_p=q^{[n(n-1)-2i(n-i)]/2}$.  Then the inequality $|H\cap X|_p<|\Out(X)|$ implies that
	\begin{align}\label{eq:cp-1}
		q^{n(n-1)-2i(n-i)}<4\cdot a^{2}\cdot \gcd(n,q-1)^{2}.
	\end{align}
	Since $q^{n(n-2i-1)}\leq q^{n(n-1)-2i(n-i)-2}$, it follows from \eqref{eq:cp-1} that $q^{n(n-2i-1)}<4a^2$. If $n-2i\geq 2$, then $q^{n}<4a^2$, and since $q\geq 2a$, we have that $q^{n-2}<1$, which is impossible. Thus $n-2i\leq 1$. Note that $2i<n$. Then $n-2i=1$, and so $q^{(n-1)^2-4}<16a^4$. This inequality is not valid for $n\geq 4$. Therefore, $n=3$ and $i=(n-1)/2=1$, and hence
	$H$ is a $\Cmc_1$-subgroup of type $\GL_1(q)\oplus \GL_{2}(q)$. We now apply \eqref{eq:cp-1} and conclude that $q$ is $4$ or $16$.
	If $q=4$, then by \cite{b:Atlas}, we have that $2^2=|H\cap X|_2<|\Out(X)|=2^2{\cdot}3$, and if $q=16$, then
	$2^4=|H\cap X|_2<|\Out(X)|=2^3{\cdot}3$. These yield part (a.i) as claimed. \smallskip
	
	\noindent \textbf{(2)} If $H\in \Cmc_2$, then $H$ preserves a partition $V=V_1 \oplus \cdots \oplus V_t$ with each $V_i$ of the same dimension, say $m$, and so $n=mt$. Here by~\cite[Proposition 4.2.9]{b:KL-90}, we have that $H_{0}\cong \,^{\hat{}}\SL_{m}(q)^t{\cdot} (q-1)^{t-1}{\cdot} \S_{t}$ with $n=mt$. Therefore, $|H\cap X|_p\geq q^{mt(m-1)/2}$. Then the inequality $|H\cap X|_p<|\Out(X)|$ implies that $q^{mt(m-1)}<4a^{2}{\cdot}\gcd(n,q-1)^{2}$. If $m\geq3$, then as $\gcd(n,q-1)<q$, it follows that  $q^{6t-2}<q^{mt(m-1)-2}<4a^2$, and since $q\geq  2a$, we have $q^{6t-4}<1$,  which is impossible. If $m=2$, then $q^t<2a{\cdot}\gcd(2t,q-1)$, and so $q^t<2at{\cdot}\gcd(2,q-1)$. This inequality does not hold for any $q=p^a$. Therefore, $m=1$, and hence $H$ is a $\Cmc_2$-subgroup of type $\GL_1(q) \wr \S_n$, and this is part (a.ii).\smallskip
	
	\noindent \textbf{(3)} If $H\in \Cmc_3$, then $H$ is an extension field subgroup. In this case by~\cite[Proposition 4.3.6]{b:KL-90}, we have that $H_{0}\cong \,^{\hat{}}\SL_{m}(q^t){\cdot} (q^t-1)(q-1)^{-1}{\cdot} t$ with $n=mt$ and $t$ prime. It follows
	from $|H\cap X|_p<|\Out(X)|$ that $q^{mt(m-1)/2}<2a{\cdot}\gcd(n,q-1)$. Note that $t\geq2$ and $\gcd(n,q-1)<q$. If $m\geq2$, then $q^{2}<q^{mt(m-1)-2}<4a^{2}$,  which is impossible. Therefore, $m=1$, and this yields part (a.iii).\smallskip
	
	\noindent \textbf{(4)} If $H\in \Cmc_4$, then $H$ stabilises a tensor product of spaces of different dimensions. Here~\cite[Proposition 4.4.10]{b:KL-90} implies that $|H\cap X|_p\geq q^{(m^2-m+t^2-t)/2}$ with $m>t>1$. Since $t^2-t\geq2$ and $m\geq3$,  the inequality $|H\cap X|_p<|\Out(X)|$ yields $q^4<q^{m(m-1)-2}<4a^{2}$, which is impossible.\smallskip
	
	\noindent \textbf{(5)} If $H\in \Cmc_7$, then $H$ stabilises the tensor product of spaces of the same dimension, say $m$, and so $n=m^t$ with $t\geq2$ and $m\geq3$.  Here by~\cite[Proposition 4.7.3]{b:KL-90}, $|H\cap X|_p\geq q^{mt(m-1)/2}$. Hence the inequality $|H\cap X|_p<|\Out(X)|$ implies that $q^{10}<q^{mt(m-1)-2}<4a^2$, which is impossible.\smallskip
	
	\noindent \textbf{(6)} If $H\in \Cmc_8$, then $H$ is a classical group. So by~\cite[Propositions 4.8.3, 4.8.4 and 4.8.5]{b:KL-90},
	we need to consider the following cases:\smallskip
	
	\noindent (6.1)\quad $H\cap X \cong \,^{\hat{}}\Sp_{2m}(q){\cdot} \gcd(m, q-1)$ with $n=2m\geq 4$. Then $|H\cap X|_p=q^{m^2}$ and the inequality $|H\cap X|_p<|\Out(X)|$ yields $q^3<q^{m^2-1}<2a$, which is impossible.\smallskip

	\noindent (6.2)\quad $H\cap X \cong\,^{\hat{}}\SO_{2m+1}(q)$ with $m\geq 1$ and $q$ odd. In this case again, $|H\cap X|_p=q^{m^2}$. If $m\geq2$, then $q^3<2a$, which is impossible. Hence $m=1$, and this follows part (a.iv).\smallskip

	\noindent (6.3)\quad $H\cap X \cong\,^{\hat{}}\SO_{2m}^{\e}(q)$ with $m\geq 2$, $q$ odd and $\e =\pm$. Here $|H\cap X|_p=q^{m(m-1)}$. If $m\geq3$, then the inequality $|H\cap X|_p<|\Out(X)|$ implies that $q^4<q^{m(m-1)-2}<4a^2$, which is impossible. Thus $m=2$ and  this yields part (a.v).\smallskip

	\noindent (6.4)\quad ${H\cap X \cong \,^{\hat{}}\SU_{n}(q^{\frac{1}{2}}){\cdot} \gcd(n, q^{\frac{1}{2}}-1)}$ with $n\geq 3$ and $q$ a square. Here ${|H\cap X|_p}=q^{n(n-1)/4}$. If $n\geq4$, then by the inequality $|H\cap X|_p<|\Out(X)|$, we have $q^2<q^{[n(n-1)-4]/4}<2a$, which is impossible. Thus $n=3$. So $q^3<4a^2{\cdot}\gcd(3,q-1)^2$. This inequality holds only for $q=4$, and then $12=|\Out(X)|>|H\cap X|_2=2^{3}$. This is part (a.vi).
\end{proof}

We will use the following elementary lemma in number theory.

\begin{lemma}\label{lem:equ}
	Let $q$ be a prime power and $n$ be a positive integer number. Then
	\begin{enumerate}[\quad \rm (a)]
		\item  $\prod_{i=1}^{n} (q^{2i}-1)<q^{n(n+1)}$;
		\item $\prod_{i=2}^{n} (q^{i}-1)<q^{(n^2+n-2)/2}$.
	\end{enumerate}
\end{lemma}

\section{Proof of the main results}\label{sec:proof}
In this section, we prove Theorem \ref{thm:main}.
Suppose that $\Dmc$ is a nontrivial $2$-design with $\gcd (r,\lambda)=1$ and that $G$ is an almost simple automorphism group of $\Dmc$ whose socle $X$ is a finite simple classical group of Lie type. According to~\cite[Main~Theorem]{a:Saxl2002}, we need only to focus on $2$-designs with $\lambda \geq  2$ and by \cite{a:A-Exp-CP,a:ABD-Un-CP,a:Zhou-sym-sporadic,a:Zhan-CP-nonsym-sprodic,a:Zhou-CP-nonsym-alt,a:Zhuo-CP-sym-alt},
we will treat the case where $X$ is $\PSL_{n}(q)$, $\PSp_{n}(q)$ or $\POm^\e_n(q)$.
Suppose now that $G$ is flag-transitive. Then Proposition~\ref{prop:flag} implies that $G$ is point-primitive. Let $H=G_{\alpha}$, where $\alpha$ is a point of $\Dmc$. Therefore, $H$ is maximal in $G$  by \cite[Corollary 1.5A]{b:Dixon}, and so Lemma~\ref{lem:New} implies that
\begin{align}
	v=\frac{|X|}{|H\cap X|}.\label{eq:v}
\end{align}

In what follows, we discuss each possibilities for $X$ separately. We first observe that  $\PSL_{2}(q)$ is isomorphic to $\PSU_{2}(q)$, and hence Proposition~\ref{prop:psl2} below follows immediately from \cite[Theorem 1.1]{a:ABD-PSL2} and \cite[Proposition 4.2]{a:ABD-Un-CP}.

\begin{proposition}\label{prop:psl2}
	Let $\Dmc$ be a nontrivial $2$-design with $\gcd (r, \lambda)=1$. Suppose that $G$ is an automorphism group of $\Dmc$ of almost simple type with socle $X=\PSL_{2}(q)$. If $G$ is flag-transitive, then $\Dmc$ is the Witt-Bose-Shrikhande space $\W(2^a)$ with parameters $(2^{a-1}(2^a-1),2^{a-1},1)$ for $a \geq  3$, and $X$ is $\PSL_{2}(2^a)$ or $(v,b,r,k, \lambda)$, $G$ and $G_\alpha$ are as in lines $1$-$6$ of {\rm Table~\ref{tbl:main}}.
\end{proposition}
\begin{table}
	\caption{Some maximal subgroups of $X=\PSL_n(q)$ with small $n$ and $q$.}\label{tbl:psl-c6-s}
	\begin{tabular}{cllll}
		\noalign{\smallskip}\hline\noalign{\smallskip}
		Class & $X$ & $H_{0}$ & $v$ & $u_r$  \\
		\noalign{\smallskip}\hline\noalign{\smallskip}
		$\Cmc_{6}$ &
		$\PSL_3(7)$ &
		$3^2.\Q_8$ &
		$2^{2}{\cdot}7^3{\cdot}19$ &
		$3$
		\\
		$\Cmc_{6}$ &
		$\PSL_4(5)$ &
		$2^4.\A_{6}$ &
		$5^{5}{\cdot} 13{\cdot} 31$ &
		$2$
		\\
		$\Smc$ &
		$\PSL_{3}(4)$ &
		$\A_{6}$ &
		$2^{3}{\cdot} 7$ &
		$5$
		\\
		$\Smc$ &
		$\PSL_{4}(7)$ &
		$\PSU_{4}(2)$ &
		$2^{3}{\cdot} 5{\cdot} 7^{6}{\cdot} 19$ &
		$3$
		\\
		\noalign{\smallskip}\hline\noalign{\smallskip}
	\end{tabular}
\end{table}

\begin{proposition}\label{prop:psl}
	Let $\Dmc$ be a nontrivial $2$-design with $\gcd(r, \lambda)=1$. Suppose that $G$ is an automorphism group of $\Dmc$ of almost simple type with socle $X=\PSL_{n}(q)$ with $ n\geq 3$ and $(n,q)\neq (3,2)$ and $(4,2)$. If $G$ is flag-transitive and $H=G_{\alpha}$ with $\alpha$ a point of $\Dmc$, then $H\cap X\cong \,^{\hat{}}[q^{n-1}]{:}\SL_{n-1}(q){\cdot} (q-1)$ is a parabolic subgroup, $v=(q^{n}-1)/(q-1)$ and $r$ divides $(q^{n}-q)/(q-1)$. In particular, if $\lambda =1$, then $\Dmc$ is the Desarguesian plane. Moreover, if $\Dmc$ is symmetric, then $\Dmc$ is the projective space $\PG_{n-1}(q)$.
\end{proposition}

\begin{proof}
	Suppose that $X=\PSL_n(q)$ with $n\geq3$ and $q=p^a$. We note here that the almost simple groups with socle $\PSL_{3}(2)\cong \PSL_{2}(7)$ and $\PSL_{4}(2)\cong \A_{8}$ have been treated in Proposition~\ref{prop:psl2} and \cite{a:Zhou-CP-nonsym-alt}. Therefore, we will exclude these cases in our arguments below. Let now $H_{0}=H\cap X$, where $H=G_{\alpha}$ with $\alpha$ a point of $\Dmc$. Since $H$ is maximal in $G$, by Aschbacher's Theorem~\cite{a:Aschbacher-84}, $H$ belongs to a collection $\Cmc_i$ or $\Smc$, for some $i=1, \ldots,8$. If $H$ is neither a $\Cmc_6$-subgroup, nor a $\Smc$-subgroup, then  Lemma~\ref{lem:coprime}(a) gives the list of possible subgroups $H$ satisfying $|H_0|_p<|\Out(X)|$. On the other hand, if $H$ is not a parabolic subgroup satisfying $|H_0|_p\geq |\Out(X)|$, then since $|G|<|H|\cdot|H|^2_{p'}$ by Corollary~\ref{cor:large-2}, we conclude that $|X|\leq |H\cap X|^{3}$, and
	the possibilities for such $H$ are recorded in~\cite[Theorem 7 and Proposition 4.7]{a:AB-Large-15}. In conclusion, we have one of the following possibilities:
	
	\begin{enumerate}[\rm (1)]
		\item $H \in \Cmc_{1}\cup\Cmc_{6}\cup\Cmc_{8}\cup \Smc$;
		\item $H$ is a $\Cmc_2$-subgroup of type $\GL_{m}(q)\wr\S_t$ with $m=1$ or $t=2,3$, where $m=n/t$;
		\item $H$ is a $\Cmc_3$-subgroup of type $\GL_{m}(q^t)$ with $m=1$ or $t=2,3$, where $m=n/t$;
		\item $H$ is a $\Cmc_5$-subgroup of type $\GL_{n}(q_{0})$ with $q=q_{0}^{t}$ and $t=2, 3$.
	\end{enumerate}
	In what follows, we analyse each of these possible cases.\smallskip
	
	\noindent\textbf{(1)} Let $H$ be a $\Cmc_{1}$-subgroup. In this case $H$ is reducible, that is, $H \cong P_{i}$ stabilises a subspace of $V$ of dimension $i$ with $1\leq  i \leq n/2$ or $G$ contains a graph automorphism and $H$ stabilises a pair $\{U,W\}$ of subspaces of dimension $i$ and $n-i$ with $i<n/2$.\smallskip
	
	\noindent{(a)}\quad  Suppose first that $H\cong P_{i}$ for some $1\leq  i \leq n/2$. Then by~\cite[Proposition 4.1.17]{b:KL-90}, we have that
	$H_{0}\cong \,^{\hat{}}[q^{i(n-i)}]{:}\SL_{i}(q){\times }\SL_{n-i}(q){\cdot}(q-1)$.
	Thus by \cite[Corollary~1]{a:Korableva-LnUn} , we observe that
	\begin{align}\label{eq:psl-c1}
		v=\frac{(q^n-1)\cdots(q^{n-i+2}-1)(q^{n-i+1}-1)}{(q^i-1)(q^{i-1}-1)\cdots(q^2-1)(q-1)}>q^{i(n-i)}.
	\end{align}
	Let $H=P_1$. Note that the group $G$ is $2$-transitive in this action and $v=(q^{n}-1)/(q-1)$. Moreover, $r$ divides $v-1=(q^{n}-q)/(q-1)$ which is also the  nontrivial subdegree of $X$. In particular, if $\lambda =1$, then $\Dmc$ is the  Desarguesian plane. Moreover, if $\Dmc$ is symmetric, then by \cite{a:Kantor-85-2-trans} , $\Dmc$ is the projective space $\PG_{n-1}(q)$.
	
	Suppose $i>1$. Then by Lemmas~\ref{lem:subdeg-PSp} and~\ref{lem:six}(d), we see that $r$ divides the subdegree $d$  which is $q(q^i-1)(q^{n-i}-1)/(q-1)^2$. This implies that $r<q^{n-1}$ if $q\neq2$, and $r<2^{n+1}$ if $q=2$. Since $v>q^{i(n-i)}$, it follows from the fact $\lambda v< r^2$ that $i=2$ or $(i,n)\in(3,6)$ if $q\neq 2$, and $i=2$ or $(i,n)\in \{(3,6),(3,7),(3,8),(3,9),(4,8)\}$ if $q=2$. Let $i\geq3$. We now apply Lemma~\ref{lem:six}(a) and (d) which says that $r$ divides $\gcd(v-1,d)$. Thus we can find a lower bound $l_v$ for $v$ as in the fourth column of Table~\ref{tbl:psl-c1-lv-ur} and an upper bound $u_r$ for $r$ as in the fifth column of Table~\ref{tbl:psl-c1-lv-ur}, and then we easily observe that $\lambda v\geq l_{v}>u_{r}^2\geq r^{2}$, which is a contradiction.
	\begin{table}
		\caption{Parameters $l_v$ and $u_r$ in Proposition~\ref{prop:psl}.}\label{tbl:psl-c1-lv-ur}
		\begin{tabular}{cllll}
			\noalign{\smallskip}\hline\noalign{\smallskip}
			Class & $i$ & $n$ & $l_v$ & $u_r$  \\
			\noalign{\smallskip}\hline\noalign{\smallskip}
			$\Cmc_{1}$ &
			$3$ &
			$6$ &
			$q^9$ &
			$q$
			\\
			$\Cmc_{1}$ &
			$3$ &
			$6$ &
			$2^9$ &
			$2$
			\\
			$\Cmc_{1}$ &
			$3$ &
			$7$ &
			$2^{12}$ &
			$2{\cdot}5$
			\\
			$\Cmc_{1}$ &
			$3$ &
			$8$ &
			$2^{15}$ &
			$2{\cdot}31$
			\\
			$\Cmc_{1}$ &
			$3$ &
			$9$ &
			$2^{18}$ &
			$2{\cdot}3$
			\\
			$\Cmc_{1}$ &
			$4$ &
			$8$ &
			$2^{16}$ &
			$2$
			\\
			\noalign{\smallskip}\hline\noalign{\smallskip}
		\end{tabular}
	\end{table}
	
	Therefore, $i=2$. In this case, $G$ is rank 3. If $\Dmc$ is symmetric, then the possibilities for $\Dmc$ can be read off from \cite{a:Dempwolff2001}, but we have no example with condition $\gcd(k,\lambda)=1$. In what follows, we assume that $\Dmc$ is nonsymmetric, that is to say, $v<b$, or equivalently, $k<r$. Here $v=(q^n-1)(q^{n-1}-1)/(q^2-1)(q-1)$, and the nontrivial subdegrees of $G$ listed in Table~\ref{tbl:subdeg-PSp}.
	So Lemma~\ref{lem:six}(d) implies that $r$ divides
	\begin{align}\label{eq:psl-para-subdeg}
		\frac{q(q^{n-2}-1)}{(q-1)}{\cdot}\gcd \left(q+1,\frac{(q^{n-3}-1)}{(q^2-1)}\right).
	\end{align}
	
	If $n$ is even, then by~\eqref{eq:psl-para-subdeg}, we have that
	\begin{align*}
		r \text{ divides } &\, \frac{q(q^{n-2}-1)}{(q-1)} {\cdot}\gcd \left(q+1,\frac{(q^{n-4}+ \cdots+q+1)}{(q+1)}\right),
	\end{align*}
	and so $r$ divides $q(q^{n-2}-1){\cdot}\gcd((q+1)^2,(q^{n-4}+ \cdots+q+1)/(q^{2}-1)$.
	Since $\gcd(q+1,q^{n-4}+ \cdots+q+1)=1$, $r$ must divide $q(q^{n-2}-1)/(q^2-1)$.
	We now apply Lemma~\ref{lem:six}(c) and conclude that $q^{2n-4}<v<r^2<4q^{2n-6}$, which is impossible.
	
	Therefore, $n$ is odd. Here by~\eqref{eq:psl-para-subdeg}, $r$ divides $cf(q)$, where
	$c=\gcd\left(q+1,(n-3)/2\right)$ and  $f(q)=q(q^{n-2}-1)/(q-1)$.
	Assume first that $n=5$. Then $v=(q^2+1)(q^4+q^3+q^2+q+1)$, and so $r=q(q^2+q+1)/m$, for some positive integer $m$. But the inequality $\lambda v<r^2$ does not hold for $\lambda \geq 2$.
	Assume now $n\geq7$. Let
	\begin{align*}
		h(q)=& \frac{q^{n}-1}{q-1}=q^{n-1}+q^{n-2}+\cdots+q+1,\\
		g(q)= &\frac{q^{n-1}-1}{q^2-1}=q^{n-3}+q^{n-5}+\cdots+q^2+1 ,\\
		f_1(q)= & \frac{q^{n-2}-1}{q-1}=q^{n-3}+q^{n-4}+\cdots+q+1\\
	\end{align*}
	and
	\begin{align}\label{eq:psl-para-dq}
		d(q)=& \frac{v-1}{f(q)}=\frac{q^n+q^2-q-1}{q^2-1}=q^{n-2}+q^{n-4}+\cdots+q^3+q+1.
	\end{align}
	Thus $v=g(q){\cdot}h(q)$ and $f(q)=q\cdot f_{1}(q)$. Since $r$ divides $cf(q)$, there exists $m$ such that $mr=cf(q)$. Hence $r=cf(q)/m$. Since $r^2>v$, it follows  that
	\begin{align}\label{eq:para-psl-m}
		m\leq 2q.
	\end{align}
	Moreover, by Lemma~\ref{lem:six}(a) and \eqref{eq:psl-para-dq}, we have that
	\begin{align*}
		k=\frac{m\lambda (v-1)}{cf(q)}+1=\frac{m\lambda d(q)+c}{c}.
	\end{align*}
	Since $k< r$, we have that $m\lambda d(q)/c<cf(q)/m$, and this yields
	\begin{align*}
		m^2\lambda <c^2\frac{f(q)}{d(q)}= & \, c^2\frac{q(q^{n-2}-1)(q^2-1)}{(q-1)(q^n+q^2-q-1)}
		=  \, c^2\frac{q(q^{n-2}-1)(q+1)}{q^n+q^2-q-1} \\
		< & \, c^2\frac{q(q^{n-2}-1)(q+1)}{q^n}
		= c^2\frac{(q^{n-2}-1)(q+1)}{q^{n-1}}.
	\end{align*}
	So $m^2 \lambda<c^2(q+1)/q$. If $m\geq1$, then $\lambda<c^2(q+1)/ q$.
	Note by~\cite{a:Liebeck1985},
	\begin{align*}
		h(q)=\frac{q^n-1}{q-1}\, \text{ divides } \, b=\frac{vr}{k}=\frac{c^2q {\cdot}f_1(q){\cdot}g(q){\cdot}h(q)}{m(m\lambda d+c)}.
	\end{align*}
	Therefore, $m\lambda d+c$ must divide $m \lambda c^2q{\cdot}f_1(q){\cdot}g(q)$. We first note that $q{\cdot}g(q)=q^{n-2}+q^{n-4}+\cdots+q^3+q=d(q)-1$. From this we observe that
	\begin{align}\label{eq:psl-para-1}
		\gcd(m\lambda d(q)+c,m\lambda q{\cdot}g(q))=\gcd(m\lambda d(q)+c,m\lambda +c).
	\end{align}
	We next obtain
	$q^2f_1(q)=(q+1)d(q)-2q-1$. Then
	\begin{align*}
		m \lambda q^2f_1(q)= &  m \lambda (q+1)d(q)-2 m \lambda q- m \lambda  \\
		= & (q+1)( m \lambda d(q)+c)-c(q+1)-2 m \lambda q- m \lambda  \\
		= & (q+1)( m \lambda d(q)+c)-q(2 m \lambda +c)- m \lambda -c.
	\end{align*}
	Thus
	\begin{align}\label{eq:psl-para-2}
		\gcd(m\lambda d(q)+c,m\lambda q^2f_1(q))=\gcd(m\lambda d(q)+c,q(2m\lambda +c)+m\lambda +c).
	\end{align}
	Therefore,~\eqref{eq:psl-para-1} and~\eqref{eq:psl-para-2} yield
	$m\lambda d(q)+c$  divides $c^2(m\lambda +c)[(2m\lambda +c)q+m\lambda +c]$.
	Since  $\lambda<c^2(q+1)/q$ and by~\eqref{eq:para-psl-m}, we have $1\leq m \lambda <2c^2(q+1)$.
	Then we get
	\begin{align}\label{eq:psl-para-divid}
		d(q)+c<c^2[2c^2(q+1)+c][\left( 4c^2(q+1)+c \right) q+ 2c^2(q+1)+c].
	\end{align}
	We now apply the fact that $c\leq c^2$ for the last inequality and so
	\begin{align*}
		d(q)  & <  c^6(8q^3+26q^2+27q+9)=c^6(q+1)(2q+3)(4q+3).
	\end{align*}
	From~\eqref{eq:psl-para-dq}  and since $q^2-q-1>0$, we have $q^n/(q^2-1)<c^6(q+1)(2q+3)(4q+3)$.
	Note that $q+1<2q$, $2q+3<4q$ and $4q+3<6q$. Then $q^n<c^6(q^2-1)(q+1)(2q+3)(4q+3)<2^4{\cdot}3{\cdot}c^6q^5$. Since $c\leq (n-3)/2$, we conclude that  $q^{n-5}<3(n-3)^6/4$.
	This inequality holds only when $n=7$ and $q\leq 53$, $n=9$ and $q\leq 13$,  $n=11$ and $q\leq 7$, $n=13$ and $q\leq 5$, $n=15$ and $q\leq 4$,  $n=17,19$ and $q\leq 3$, or $21\leq n\leq 33$ and $q\leq 2$. But then again applying~\eqref{eq:psl-para-divid}, we obtain $(n,q)=(7,q)$ with $q=2,3,5,7,9,11,13,17,19,23$, or $(n,q)\in\{(9,2),(9,5),(9,8),(11,3),(13,4),(15,2)\}$.  These remaining cases can be easily ruled out, as there are no parameters $(v,b,r,k,\lambda)$  satisfying the conditions $r(k-1)=\lambda(v-1)$ and $bk=vr$.
	\smallskip
	
	\noindent{(b)}  Suppose now that $H$ stabilises a pair $\{U,W\}$ of subspaces of dimension $i$ and $n-i$ with $i<n/2$. Assume first that $U\subset W$. Here by Lemma~\ref{lem:subdeg}, there is a subdegree which is a power of $p$.
	On the other hand, the highest power of $p$ dividing $v-1$ is $q$ if $p$ is odd, it is $2q$ if $q>2$ is even, and it is at most $2^{n-1}$ if $q=2$. Hence $r^2<v$, which is a contradiction.
	Assume now that $U\cap W=0$. Then by~\cite[Proposition 4.1.4]{b:KL-90}, $H_{0}\cong (\,^{\hat{}}\SL_{i}(q)\times \SL_{n-i}(q)){:}(q-1)$. Note by~\cite[Lemma 4.2 and Corollary 4.3]{a:AB-Large-15} and~\eqref{eq:v} that $v>q^{2i(n-i)}$. Also
	by  Lemma~\ref{lem:Tits}, $p$ divides $v$, and so $r$ is coprime to $p$ by Lemma~\ref{lem:six}(a). If $i=1$, then Lemmas~\ref{lem:subdeg-PSp} and~\ref{lem:six}(d) imply that $r$ divides $(q^{n-1}-1)/(q-1)$, whereas $\gcd(r,p)=1$. Note that $(q^{n-1}-1)/(q-1)<2q^{n-2}$. Thus  $r<2q^{n-2}$. We apply Lemma~\ref{lem:six}(c) and deduce that $q^{2n-2}\leq \lambda v<r^2<4q^{2n-4}$, that is to say, $q<2$, which is impossible. Therefore $i>1$. According to~\cite[p. 339-340]{a:Saxl2002}, there is a subdegree of $G$ with the $p'$-part dividing $(q^i-1)(q^{n-i}-1)$, and so $r<2q^n$. Since $n>2i$, it follows that $v> q^{2i(n-i)}>4q^{2n}>r^2$, which is a contradiction.\smallskip
	
	Let $H$ be a $\Cmc_6$-subgroup. Then by~\cite[Propositions 4.6.5 and 4.6.6]{b:KL-90} and the inequality $|G|<|H|{\cdot}|H|^2_{p'}$,
	we need only to consider the pairs $(X,H_0)$ listed in Table~\ref{tbl:psl-c6-s}. For
	each such $H_0$, by~\eqref{eq:v}, we obtain $v$ as in the fourth column of Table~\ref{tbl:psl-c6-s}. Moreover,
	Lemma~\ref{lem:six}(a)-(c) says that $r$ divides $\gcd(|H|,v-1)$, and so we can find an upper bound
	$u_r$ of $r$ as in the fifth column of Table~\ref{tbl:psl-c6-s}. Then the inequality $\lambda v<r^2$ rules out
	these two possibilities.
	
	Let now $H$ be a $\Cmc_{8}$-subgroup. In this case $H$ is a classical group. Then by~\cite[Propositions 4.8.3, 4.8.4 and 4.8.5]{b:KL-90}, $H_{0}$ is isomorphic to one of the following groups:
	\begin{enumerate}[\quad(a)]
		\item $\,^{\hat{}}\Sp_{n}(q){\cdot} \gcd(n/2, q-1)$ with $n=2m\geq 4$;
		\item $\,^{\hat{}}\SO_{n}^{\e}(q)$ with $q$ odd, and ${n=2m+1\geq3}$  if $\e=\circ$ and ${n=2m\geq4}$ if $\e=\pm$;
		\item $\,^{\hat{}}\SU_{n}(q_0){\cdot} \gcd(n, q_0-1)$ with $q=q_0^2$ and $n\geq 3$.
	\end{enumerate}
	In these three cases, we apply Lemma~\ref{lem:divisible}, and as $r$ is coprime to $p$, we conclude that $r$ is divisible by a parabolic index in $H_0$.\smallskip
	
	\noindent{(a)} Here $H_0$ is a symplectic group, with $n=2m\geq4$.
	Note by \cite[Lemma 5]{a:Korableva-Sp} that the index of a parabolic subgroup of $\,^{\hat{}}\Sp_{n}(q)$ is
	\begin{align}\label{eq:psp-para-index}
		\frac{(q^{n}-1)(q^{n-2}-1)\cdots (q^{n-2i+2}-1)}{(q^i-1)(q^{i-1}-1) \cdots (q-1)},
	\end{align}
	where $i\leq n/2$. If $n=4$, then by~\eqref{eq:v}, we have $v=q^2(q^3-1)/\gcd(2,q-1)$. Also by~\eqref{eq:psp-para-index} and Lemma~\ref{lem:divisible}, we see that $(q^4-1)/(q-1)$ divides $r$, and so $r$ is divisible by $q^2+1$, but $\gcd(v-1,q^2+1)$ divides $\gcd(q^2+1,q)$ or $\gcd(q^2+1,q-1)$. Hence $\gcd(v-1,q^2+1)\leq2$, which is a contradiction. If $n=6$, then again by~\eqref{eq:v}, $v=q^6(q^5-1)(q^3-1)/\gcd(3,q-1)$ and from ~\eqref{eq:psp-para-index} and Lemma~\ref{lem:divisible}, $q^3+1$ divides $r$. Now Lemma~\ref{lem:six}(a) implies that $q^3+1$ divides $v-1$, where $v-1=q^{14}-q^{11}-q^{9}+q^{6}-1$ or $v-1=(q^{14}-q^{11}-q^{9}+q^{6}-3)/3$. In the former case, since $\gcd(v-1,q^{3}+1)$ divides $2q^{2}+1$ and $q^{3}+1$ divides $v-1$, we conclude that $q^{3}+1$ divides $2q^{2}+1$. This yields $q^{3}\leq 2q^{2}$,
	whence $q=2$, and so $v=2^6{\cdot}7{\cdot}31$ and $r=9$,  and hence $r$ is too small to satisfy $v<r^2$, which is a contradiction.
	In the latter case where $v-1=(q^{14}-q^{11}-q^{9}+q^{6}-3)/3$, by the same manner as the previous case, we obtain no possible parameters. Thus $n\geq8$. Here by~\cite[Lemma 4.2 and Corollary 4.3]{a:AB-Large-15} and~\eqref{eq:v}, we have that $v>q^{(n^2-n-6)/2}$. It follows from Lemmas~\ref{lem:subdeg-PSp} and \ref{lem:six}(d) that $r$ divides the odd part of $(q^n-1)(q^{n-2}-1)$, and certainly $r$ divides $(q^n-1)(q^{n-2}-1)/(q-1)^2$ as $r$ is coprime to $q-1$, and since $(q^n-1)/(q-1)<2q^{n-2}$, we have that $r<4q^{2n-4}$. Recall that $v>q^{(n^2-n-6)/2}$. Then  the inequality $\lambda v<r^2$ forces $n^{2}-9n-6<0$, and since $n$ is even we get $n=8$.
	Then again by~\eqref{eq:v}, $v=q^{12}(q^7-1)(q^5-1)(q^3-1)/\gcd(4,q-1)$. Hence Lemma~\ref{lem:six}(c) implies that $\lambda q^{27}/32\leq\lambda v<r^2<16q^{24}$. This yields $\lambda q^3<2^9$, whence $q\leq 5$. For these values of $q$, since $r\leq \gcd(v-1,(q^8-1)(q^6-1)/(q-1)^2)$, it follows that $r$ is at most $9$, $5$, $1$  or $3$, respectively for $q=2$, $3$, $4$ or $5$. These cases can be ruled out by Lemma~\ref{lem:six}(c).\smallskip
	
	\noindent{(b)} In this case, $H_{0}$ is of orthogonal type with $q$ is odd and $v$ is even. By the fact that $\gcd(r,v)=1$, we deduce that $r$ is odd.
	If $n=4$ and $H_0$ is of type $\O_4^{+}$, then by~\eqref{eq:v}, $v=q^4(q^3-1)(q^2+1)>q^9/2$. It follows from Lemma~\ref{lem:six}(a) and (c) that $r$ divides $a(q^2-1)^2$, and since both $q$ and $r$ are odd, we conclude that $r$ divides $a(q^2-1)^2/16$. Hence $r^2<q^9/256<q^9/2<v$, which is a contradiction.
	In the remaining cases, Lemma~\ref{lem:divisible} implies that $r$ is divisible by a parabolic index in $H$, and so Lemma~\ref{lem:odd-deg} yields $r$ is even, which is impossible.\smallskip
	
	\noindent{(c)} Here $H_0$ is of unitary type, and by~\cite[Lemma 4.2 and Corollary 4.3]{a:AB-Large-15} and~\eqref{eq:v}, we have $v>q_0^{n^2-4}$.
	If $n\geq4$, then we apply Lemmas~\ref{lem:subdeg-PSp} and~\ref{lem:six}(d) and conclude that $r$ divides
	\begin{align}\label{eq:psl-c8-unitary}
		(q_0^n-(-1)^n)(q_0^{n-1}-(-1)^{n-1}).
	\end{align}
	Hence $r<2q_0^{2n-1}$. It follows from Lemma~\ref{lem:six}(c) that $q_0^{n^2-4}\leq\lambda v<r^2<4q_0^{4n-2}$, and so $q_{0}^{n^2-4n-2}<4$, that is to say, $n^{2}-4n-4<0$, whence $n=4$. Then by~\eqref{eq:v},  $v=q_0^6(q_0^4+1)(q_0^3-1)(q_0^2+1)/\gcd(q_0-1,4)$. Again by \eqref{eq:psl-c8-unitary}, $r$ divides $(q_0^4-1)(q_0^3+1)$ and by Lemma~\ref{lem:six}(a), $r$ is coprime to $(q_0^2+1)(q_0-1)$.
	From this, we see that $r$ divides $(q_0^3+1)(q_0+1)$, and this yields $\lambda v\geq r^2$, a contradiction. Therefore, $n=3$ and $v=q_0^3(q_0^3-1)(q_0^2+1)/\gcd(q_0-1,3)$. Here the index of the parabolic subgroup of $H$  is $(q_0^3+1)$, so Lemma~\ref{lem:divisible} implies that $r$ is divisible by $q_0^2-q_0+1$. Since $r$ divides $v-1$, $q_0^2-q_0+1$ must divide $q_0^8+q_0^6-q_0^5-q_0^3-1$, and this forces $q_0=2$. Hence $v=280$ and $r\leq \gcd(v-1,|H|)=9$, again contrary to $\lambda v<r^2$.\smallskip
	
	Let finally $H$ be a $\Smc$-subgroup. Then by~\cite[Theorem 4.2 and Corollary 4.3]{a:Liebeck1985}, and so either $|H|<q^{2n+4}$, or
	$X$ and $H_0$ are as in \cite[Table 4]{a:Liebeck1985}. In the former case, if $|H|<q^{2n+4}$, then Corollary~\ref{cor:large} and~\cite[Corollary 4.3]{a:AB-Large-15} imply that $q^{n^2-2}<|G|<|H|^3<q^{6n+12}$, and this yields $n^{2}-6n-14<0$,
	so $n\leq7$. Thus for $n\leq7$, the subgroups $H$ are listed in~\cite[Chapter 8]{b:BHR-Max-Low}. Since $|G|<|H|{\cdot}|H|^2_{p'}$, we only need to consider the pairs $(X,H_0)$ listed in Table~\ref{tbl:psl-c6-s}. For each such $H_0$, as before, the inequality $\lambda v<r^2$ does not hold.
	In the latter case, $n$ is $d(d-1)/2$, $27$, $16$, or $11$ and $H_0$ is $\PSL_d(q)$, $\E_6(q)$, $\POm^+_{10}(q)$ or $\M_{24}$, respectively. For the case $n=d(d-1)/2$ by~\cite[Theorem 4.1]{a:Liebeck1985}, $|H|<q^{3n}$.  As $q^{n^2-2}<|G|<|H|^3<q^{9n}$, we have $n\leq9$, and so the inequality $3\leq d(d-1)/2\leq 9$ implies that $d=3$ or $4$, and so $n=3$ or $6$, respectively, but this gives no  possible parameters. The remaining cases also can be ruled out by Corollary~\ref{cor:large}.\smallskip

	\noindent\textbf{(2)} Let $H$ be  a $\Cmc_2$-subgroup of type $\GL_{m}(q)\wr\S_t$ with $m=1$ or $t=2,3$, where $m=n/t$. Then by~\cite[Proposition 4.2.9]{b:KL-90}, we have that $H_{0}\cong \,^{\hat{}}\SL_{m}(q)^t{\cdot} (q-1)^{t-1}{\cdot} \S_{t}$ with $n=mt$.
	It follows from~\eqref{eq:v} and Lemma~\ref{lem:equ} that $v>q^{n(n-m)}/(t!)$.
	If $m=1$, then $n=t$ and by Lemmas~\ref{lem:subdeg-PSp} and~\ref{lem:six}(d), $r$ divides $2n(n-1)(q-1)$. Lemma~\ref{lem:six}(c) implies that $q^{n(n-1)}/(n!)\leq\lambda v<r^2\leq 4n^2(n-1)^2(q-1)^2$, and so $q^{n(n-1)}/(n!)< 4n^2(n-1)^2(q-1)^2$. This inequality holds only for $(n,q)\in \{ (3,2),(3,3),(3,4),(4,2) \}$, and considering the fact that $(n,q)\neq (3,2)$ and $(4,2)$, we have $(n,q)=(3,3)$ or $(3,4)$. Note that the latter case can be  ruled out, as $v$ is not integer.
	For the first case, we obtain $v=234$, and  this is impossible since $r$ divides $\gcd(v-1,2n(n-1)(q-1))=\gcd(233,24)=1$.
	If  $m>1$, then again by Lemma~\ref{lem:subdeg-PSp}, the parameter $r$ divides $t(t-1)(q^m-1)^2/(q-1)$, and so $r<2q^{2m-1}t^2$. It follows from Lemma~\ref{lem:six}(c) that $q^{n(n-m)}/(t!)\leq\lambda v<r^2<4q^{4m-2}t^4$, and so  $q^{n(n-m)-4m+2}<4t^4(t!)$ for $t=2,3$ and $n=mt$. This inequality is true for  $n=4$ and $q\in \{2,3,4,5,7,8,9,11\}$, but for these cases, we cannot find any possible parameters by Lemma~\ref{lem:six}. \smallskip
	
	\noindent\textbf{(3)} Let $H$ be  a $\Cmc_3$-subgroup of type $\GL_{m}(q^t)$ with $m=1$ or $t=2,3$, where $m=n/t$. Then by~\cite[Proposition 4.3.6]{b:KL-90}, we have that $H_{0}\cong \,^{\hat{}}\SL_{m}(q^t){\cdot} (q^t-1)(q-1)^{-1}{\cdot} t$ with $n=mt$.
	We apply Corollary~\ref{cor:large-2}, and conclude that $|X|< |\Out(X)|^2{\cdot} |H_{0}| {\cdot} |H_{0}|_{p'}^{2}$, and since $2a\leq q$ and $\gcd(n,q-1)\leq q-1$, we also know that $|\Out(X)|\leq q(q-1)$. By~\cite[Lemma 4.2 and Corollary 4.3]{a:AB-Large-15}, we have that
	\begin{align*}
		q^{m^2t^2-2}<t^{3}q^{m^2t+3}(1-q^{-2t}){\cdot}\prod_{i=2}^{m}(q^{it}-1)^{2}(q^t-1)^{2}(q-1)^{-2}.
	\end{align*}
	Then Lemma~\ref{lem:equ}(b) implies that $q^{mt(mt-2m-1)-3}/(1-q^{-2t})<t^{3}$, and so
	\begin{align}\label{eq:psl-c3}
		q^{mt(mt-2m-1)+2t-3}<t^{3}(q^{2t}-1).
	\end{align}
	
	Let first $t=3$. Then \eqref{eq:psl-c3} implies that $q^{m^{2}-m-1}<3$, and so $m\in\{1,2\}$. If $m=2$, then by \eqref{eq:psl-c3}, we must have $q=2$, and so $v=1904640$, and $r$ divides $\gcd(v-1, |H|)=1$, which is a contradiction. Therefore $m=1$, and hence $n=3$. Thus  ${H_0\cong \,^{\hat{}}(q^2+q+1){\cdot}3<\PSL_3(q)=X}$. In this case, $v=q^3(q^2-1)(q-1)/3$ is even, and so by Lemma~\ref{lem:six}(a), $r$ is odd. By the fact that $r$ divides both $v-1$ and $|H|$, we deduce that $r$ divides $3a(q^2+q+1)$.  It follows from Lemma~\ref{lem:six}(c) that $q^3(q^2-1)(q-1)/3 \leq \lambda v<r^2\leq9a^2(q^2+q+1)^2$, and so $q^3(q^2-1)(q-1) <27a^2(q^2+q+1)^2$.
	This inequality holds only for $q \in \{3,4,5,8,9,16\}$ as $(n,q)\neq(3,2)$.
	For each such values of $q$, we obtain $v$ and an upper bound $u_r=\gcd(v-1,|H|)$ of $r$ as below
	\begin{align*}
		(v,u_r,q)\in \{ & (144,13,3),(960,7,4),(4000,93,5),(75264,73,8),(155520,91,9)\\
		&(5222400,91,16)\}.
	\end{align*}
	Hence, we observe that the condition $v<r^2$ does not hold except for $q=3,5$. In these two cases, we apply Lemma~\ref{lem:six}(a)-(d) and conclude that $r=13$ and $\lambda=1$ if $q=3$, and $r=93$ and $\lambda=1,2$ if $q=5$. Since $\lambda \geq 2$, we only consider the case where $q=5$ and $\lambda=2$, but here there is no  possible parameters satisfying Lemma~\ref{lem:six}(a) and (b).
	
	Let now $t=2$. Then $n$ is even, and so $H_{0}\cong \,^{\hat{}}\SL_{m}(q^2){\cdot} (q+1){\cdot}2$. By~\eqref{eq:v}, $v>q^{n(n-m)}/2$ with $m=n/2$.  Note by  Lemma~\ref{lem:Tits} that $p$ divides $v$, and so Lemma~\ref{lem:six}(a) implies that $r$ is coprime to $p$. Suppose first that $n\geq 8$. It follows from Lemmas~\ref{lem:subdeg-PSp} and~\ref{lem:six}(d) that $r<q^{2n-2}$. Lemma~\ref{lem:six}(c) implies that
	$q^{n^2/2}/2<v<r^2<q^{4n-4}$, and this yields $q^{n^2-8n+8}<4$, which is impossible.
	Let now $n=6$. Then $v=q^9(q^5-1)(q^3-1)(q-1)/2$. As $v$ is even, we deduce that the parameter $r$ must be odd. Lemmas~\ref{lem:New} and~\ref{lem:six}(c) yield $r$ divides $a(q+1)(q^4-1)(q^6-1)$. But since $v$ is divisible by $q-1$ and $q^3-1$, and since $\gcd(v-1,q+1)$ divides $3$, we conclude that $r$ must divide $9a(q^2+1)(q^3+1)$. Thus the inequality $\lambda v<r^2$ yields $q^{18}/2\leq \lambda v<r^2\leq9^2a^2(q^2+1)^2(q^3+1)^2$, then $q^{18}<2\cdot 9^2a^2(q^2+1)^2(q^3+1)^2$, and so $q=2$. Then $v=55552$, and so $r\leq \gcd(v-1,|H|)=3$, and this contradicts Lemma~\ref{lem:six}(c). Hence $n=4$. Then by~\eqref{eq:v}, $v=q^4(q^3-1)(q-1)/2$ is even. Since $\gcd(r,v)=1$, we have $r$ is odd and coprime to $q-1$. Here $|H_0|=2q^2(q-1)(q+1)^2(q^2+1)/\gcd(4,q-1)$. So by Lemma~\ref{lem:six}(a) and (c) and the fact that $\gcd(v-1, q+1)=1$, we deduce that $r$ is also coprime to $q+1$. Thus $r$ divides $a(q^2+1)$.  Again, applying Lemma~\ref{lem:six}(c), we conclude that $q^8/2\leq \lambda v<r^2<q^6/4$, which is impossible.\smallskip
	
	\noindent\textbf{(4)} Let $H$ be  a $\Cmc_5$-subgroup of type $\GL_{n}(q_{0})$ with $q=q_{0}^{t}$ and $t=2, 3$. Then by \cite[Proposition 4.5.3]{b:KL-90}, we see that $H_{0}\cong\,^{\hat{}} \SL_{n}(q_{0}){\cdot} \gcd (n, (q-1)/(q_{0}-1))$ with $q=q_{0}^{t}$ and $t=2,3$.
	
	Let first  $t=3$. We apply Corollary~\ref{cor:large-2}, Lemma~\ref{lem:equ}(b), \cite[Corollary 4.3]{a:AB-Large-15} and the fact that $a^2\leq 2q$, and by the same argument as in case (3), we deduce that $q_0^{n^2-n-6}/(1-q_0^{-2})<8n^5$, or equivalently,
	\begin{align}\label{eq:psl-c5}
		q_0^{n^2-n-4}<8n^5(q_0^{2}-1).
	\end{align}
	This inequality implies that $n=3$, $4$ or $5$. If $n=5$, then by~\eqref{eq:psl-c5}, we must have $q_0=2$, and so $v=2^{20}{\cdot}3^2{\cdot}7^3{\cdot}13{\cdot}73{\cdot}151$, and $r$ divides $\gcd(v-1, |H|)=1$, which  contradicts Lemma~\ref{lem:six}(c). If $n=4$, then again by~\eqref{eq:psl-c5},  we have $q_0=\{2,3,4\}$.
	Thus $v =2^{12}{\cdot}3^2{\cdot}7^2{\cdot}13{\cdot}73$, $3^{12}{\cdot}7^2{\cdot}13^2{\cdot}73{\cdot}757 $ or $2^{24}{\cdot}3^3{\cdot}7^2{\cdot}13^2{\cdot}19{\cdot}73{\cdot}241$, and
	$r$ is at most $1$, $20$ or $1$, respectively. But these cases can be ruled out by Lemma~\ref{lem:six}(c).
	If $n=3$, then by \eqref{eq:v}, we have that $v=q_0^6(q_0^6+q_0^3+1)(q_0^4+q_0^2+1)$.
	Note by Lemma~\ref{lem:Tits} that $r$ is coprime to $p$. It follows from Lemmas~\ref{lem:New} and \ref{lem:six}(c) that $r$ divides $2a(q_0^2-1)(q_0^3-1)$.
	Now by Lemma~\ref{lem:six}(c), and the fact that $q\geq2a$, we have $q_0^{16}<v<r^2<q_0^{12}$, which is impossible.

	Let now $t=2$. If $n=3$, then $v=q_0^3(q_0^3+1)(q_0^2+1)/\gcd(q_0+1,3)$. Since $r$ is coprime to $p$, we apply Lemma~\ref{lem:divisible} to $H_0$ and this together with~\eqref{eq:psl-c1} implies that $r$ is divisible by $q_0^2+q_0+1$. Thus by Lemma~\ref{lem:six}(a), we conclude that $q_0^2+q_0+1$ must divide $v-1$, and so $q_0^2+q_0+1$ divides $2q_0+\gcd(q_0+1,3)$, and therefore
	$q_0=2$. Thus $v=120$. Also the inequality $q_0^2+q_0+1\leq r\leq \gcd(v-1,|H|)$ yields $r=7$, and so $r$ is too small to satisfy $\lambda v<r^2$. If $n=4$, then $v=q_0^6(q_0^4+1)(q_0^3+1)(q_0^2+1)/\gcd(q_0+1,4)$ and  again by Lemma~\ref{lem:divisible} and using~\eqref{eq:psl-c1}, we see that $q_0^2+1$ divides $r$, but $q_0^2+1$ also divides $v$, which contradicts the fact that $\gcd(r,v)=1$. Therefore, $n\geq5$. Here by~\eqref{eq:v} and \cite[Corollary 4.3]{a:AB-Large-15}, we have that $v>q_0^{n^2-2}$. It follows from Lemmas~\ref{lem:subdeg-PSp} and~\ref{lem:six}(c) that $r<q_0^{2n-1}$, and so $v>r^{2}$, which is a contradiction.
\end{proof}

\begin{proposition}\label{prop:psl-new}
	Let $\Dmc$ be a nontrivial $2$-$((q^{n}-1)/(q-1),k,\lambda)$ design with $r\mid (q^{n}-q)/(q-1)$ and $\gcd(r, \lambda)=1$ admitting flag-transitive almost simple automorphism group $G$ with socle $X=\PSL_{n}(q)$ for $n\geq 3$. Then one of the following holds:
	\begin{enumerate}[\rm (a)]
		\item $\Dmc$ is a projective space $\PG_{n-1}(q)$ as in {\rm Example~\ref{ex:proj-space}};
		\item $\Dmc=(\Pmc,\Bmc)$ is a  $2$-$((q^{n}-1)/(q-1),q,q-1)$ with $\gcd(n-1,q-1)=1$, where $\Pmc$ is the point set of $\PG_{n-1}(q)$ and $\Bmc=(B\setminus\{\alpha\})^{G}$ with $B$ a line in $\PG_{n-1}(q)$ and $\alpha$ any point in $B$.
	\end{enumerate}
\end{proposition}
\begin{proof}
	The case where $\lambda=1$ is treated in \cite{a:Saxl2002} in which case part (a) follows. In what follows, we assume that $\lambda\geq 2$ and $(n,q)\neq (3,2)$ and $(4,2)$.
	
	We first observe that the action of $X$ is flag-transitive and $2$-transitive on points. Let $(\alpha_{1},B_{1})$ be a flag of $\Dmc$. Then $(\alpha_{1}, B_{1})^{G}$ is partitioned in $B_{1}$-orbits of the same length, as $X\unlhd G$. Thus, $\Dmc_{X} =(\Pmc,B_{1}^{X})$ is a flag- transitive $2$-$(v,k,\lambda')$ design with $v=(q^{n}-1)/(q-1)$ as $X$ is $2$-transitive on $\Pmc$. The flag transitivity of
	$G$ on $\Dmc$ and $X\unlhd G$ imply $r'=\lambda'(v-1)/(k-1)$ divides $r$, and so $\lambda'$ divides $\lambda$. Therefore, $\gcd(r',\lambda')=1$ as  $\gcd(r,\lambda)=1$. Thus, we may assume that $G=X$ and $\Dmc=\Dmc_{X}$.
	
	Assume that $p$ divides $r$. Then $\lambda$ is coprime to $p$. If $\alpha_{1}$ is a point of $\Dmc$, then $G_{\alpha_{1}}=[q^{n-1}]:W\cdot (q-1)$, where $W=\SL_{n-1}(q)$, and there is a Sylow
	$p$-subgroup $S$ of $G$ lying in $G_{\alpha_{1}}$. Let $\alpha_{2}$ be any point of $\Dmc$ distinct from
	$\alpha_{1}$. 
	Then $S_{\alpha_{1},\alpha_{2}}$ fixes at least one of the $\lambda$ blocks containing $\alpha_{1}$ and $\alpha_{2}$, say $B_{1}$.
	Moreover, $|S : S_{\alpha_{1},\alpha_{2}}|$ divides $q$, and $S_{\alpha_{1},\alpha_{2}}$ contains a Sylow $p$-subgroup of $W$. Thus either $W$ fixes $B_{1}$, or $W_{B_{1}}$ lies in a maximal parabolic subgroup of $W$. The former case implies that $\langle S_{\alpha_{1},\alpha_{2}},W \rangle\leq G_{\alpha_{1},B_{1}}$, and so $r$ divides $q\cdot \gcd(n-1,q-1)$, and hence $r\leq q(q-1)$, but this violates $\lambda v<r^{2}$ for $\lambda \geq 2$. In the latter case where $W_{B_{1}}$ lies is a maximal parabolic subgroup of $W$, we have that $|B_{1}^{W}|=e\Binom{n-1}{h}_{q}$ with $e\geq 1$ and $1\leq h\leq (n-1)/2$. Since $r$ divides $(q^{n}-q)/(q-1)$ and $W_1:=[q^{n-1}]:W$ is normal in $G_{\alpha_{1}}$, we conclude that $B_{1}^{G_{\alpha_{1}}}$ is partitioned in $W_1$-orbits of length $p^{t}e\Binom{n-1}{h}_{q}$. Therefore, $h=1$ and $r=p^{s}(q^{n-1}-1)/(q-1)$ with $s\geq t>0$, and hence $(k-1)/\lambda=q/p^{s}$. If $p^{s}<q$, then $\gcd(k,p)=1$, and so $p^{s}$ divides $b=|G:G_{B_{1}}|$. It follows from $[q^{2(n-2)}]:\SL_{n-2}(q)\leq G_{\alpha_{1},B_{1}}\leq G_{B_{1}}$ that $G_{B_{1}}$ is a geometric subgroup by \cite[Theorem~5.2.4]{b:KL-90}, and hence by inspecting maximal subgroups of $G$ whose index are divisible by $p$ \cite[Chapter 4]{b:KL-90}, we conclude that  $q$ must divide $b$, and consequently $q$ divides $r$, whereas $p^{s}<q$ is the largest power of $p$ dividing $r$. Therefore, $r=v-1=(q^{n}-q)/(q-1)$ and $k=\lambda+1$. 
	
	If $n=3$, then $B_{1}\subset B$ where $B=\Fix(S_{\alpha_{1},\alpha_{2}})$ is a line of $\PG_{2}(q)$. So $\Dmc_{0}=(B,B_{1}^{G_{B}})$ is a $2$-$(q+1,\lambda+1,\lambda)$ design with $r_{0}=q$, admitting $\PSL_{2}(q)\unlhd G_{B}^{B}$ as a flag-transitive automorphism group. Moreover, $\gcd(r,\lambda)=1$ implies that $\gcd(r_{0},\lambda)=1$. If $\Dmc_{0}$ is nontrivial, then we apply Proposition~\ref{prop:psl2} and conclude that $(q,\lambda)\in\{(5,2),(7,3),(9,5)\}$. However, only $(q,\lambda)=(7,3)$ occurs as $\gcd(r,\lambda)=1$, and hence $\Dmc$ is a $2$-$(57,4,3)$ design and $G$ is $\PSL_{3}(7)$ or $\PSL_{3}(7)\cdot 3$. In this case, $b=798$, and it is easy to check by \textsf{GAP} that $G$ has two conjugacy classes of subgroups of index $798$, but none of these subgroups has an orbit of length equal to $4$, and so this case cannot occur. If $\Dmc_{0}$ is trivial, then $k=q$ and $\lambda=q-1$. Therefore, $B_{1}=B\setminus \{\alpha\}$ with $\alpha \in B$. Moreover, $q$ is even as $\gcd(r,\lambda)=1$, and we obtain part (b). 
	
	If $n\geq 4$, then $G_{\alpha_{1},\alpha_{2},B_{1}}=G_{\alpha_{1},B_{1}}=[q^{2(n-2)}]:(\SL_{n-2}(q)\cdot(q-1))\cdot(q-1)$, and so $G_{\alpha_{1},B_{1}}\leq G_{B_{1}}\leq G_{B}$, where $B$ is either a line or a hyperplane of $\PG_{n-1}(q)$ by \cite[Chapter 4 and Theorem~5.2.4]{b:KL-90}. The latter case can be ruled out by \cite[Propositins~4.1.4 and 4.1.22]{b:KL-90} as $G_{\alpha_{1},B_{1}}\leq G_{\alpha_{1},B}$. Therefore, $B$ is a line in $\PG_{n-1}(q)$ and $G_{\alpha_{1},\alpha_{2},B_{1}}\leq G_{B}$ implies that $\alpha_{1},\alpha_{2}\in B_{1}$. Therefore, $k=|G_{B_{1}}:G_{\alpha_{1},B_{1}}|=|G_{B_{1}}^{B}:G_{\alpha_{1},B_{1}}^{B}|$. Since $G_{\alpha_{1},B_{1}}^{B}$ is isomorphic to the cyclic group of order $q-1$ and $G_{\alpha_{1},B_{1}}^{B}\leq G_{B_{1}}^{B}\leq G_{B}^{B}\cong \PGL_{2}(q)$, we derive that $G_{B_{1}}^{B}\cong q:(q-1)$, $\PGL_{2}(q)$ or $\S_{4}$ and $q=5$, as $k> 2$. Therefore, $k\in\{q,q(q+1)\}$ or $(k,q)=(6,5)$, respectively. The latter case cannot occur as $k=6$ implies that $\lambda=5$ divides $r$, which is a contradiction. If $k=q(q+1)$, then $G_{B_{1}}=G_{B}$, and so $B\subset B_{1}$, but it is impossible as $\alpha_{1},\alpha_{2}\in B_{1}\cap B$. Thus $k=q$ and $\lambda=q-1$. Moreover, $\gcd(n-1,q-1)=1$ as $\gcd(r,\lambda)=1$. Since $\alpha_{1},\alpha_{2} \in B_{1}$, $k=q$ and each nontrivial orbit of $S_{\alpha_{1},\alpha_{2}}$ is of length divisible by $q$, we conclude that $B_{1}\subset \Fix(S_{\alpha_{1},\alpha_{2}})=B$, where $B$ is a line in $\PG_{n-1}(q)$, and hence $B_{1}=B\setminus \{\alpha\}$ with $\alpha\in B$. This follows part (b).
	
	Assume that $r$ is coprime to $p$. So are $b$ and $k$  as $bk=vr$ and $v=(q^{n}-1)/(q-1)$. Thus $G_{B_{1}}$ contains a Sylow $p$-subgroup $S$ of $G$ lying in a maximal
	parabolic subgroup of $G$. 
	
	Therefore, $\Binom{n}{h}_{q}$ divides $b$, where $1\leq h\leq n/2$. Since $k$ is coprime to $p$, $S$ has its fixed point $\alpha_{1}$ of $\Dmc$ contained in $B_{1}$ and any other point $S$-orbit is of length divisible by $q$. Thus $k\geq q+1$. Since now $bk=vr$, it follows that
	\begin{align*}
	\Binom{n}{h}_{q}k \text{ divides } \left(\frac{q^{n}-1}{q-1}\right)\left(\frac{q^{n-1}-1}{q-1}\right)=\Binom{n}{2}_{q}(q+1),
	\end{align*} 
	and hence $h\in\{1,2\}$. If $h=2$, then $|G:G_{B_{1}}|\geq (q^{n}-1)(q^{n-1}-1)/[(q-1)^{2}]$, and so $k\leq q+1$. Thus $k=q+1$, $r=(q^{n-1}-1)/(q-1)$, and hence $\lambda=1$, which is not our case. Therefore, $h=1$, and hence $G_{B_{1}}\leq P$, where $P=[q^{n-1}]:\SL_{n-1}(q)\cdot(q-1)$. 
	
	Let $G_{B_{1}}$ do not contain $[q^{n-1}]:\SL_{n-1}(q)$. If $(n-1,q)\neq (4,2)$, then $|P:G_{B_{1}}|\geq (q^{n-1}-1)/(q-1)$ by \cite[Table 5.2.A]{b:KL-90} as $n>3$ and $G_{B_{1}}$ contains a Sylow $p$-subgroup of $G$. So $|G:G_{B_1}|\geq (q^{n}-1)(q^{n-1}-1)/[(q-1)^{2}]$, which is a contradiction. 
	If $(n-1,q)=(4,2)$, then $X=\PSL_{5}(2)$, $v=31$ and $r\mid 30$. Since $\gcd(r,\lambda)=1$, it follows that $(v,b,r,k,\lambda)$ is $(31, 31, 10, 10, 3)$, 
	$(31, 93, 15, 5, 2)$,
	or $(31, 31, 15, 15, 7)$. Since $X$ has no subgroup of index $93$ and the subgroups of index $31$ has no orbit of length $10$, the first two possibilities do not occur, and so $\Dmc$ is a $2$-design with parameters $(31,15,7)$.  Hence $\Dmc$ is $\PG_{4}(q)$ as in part (a) by \cite{a:ABD-PrimeLam} and \cite[Table~II.1.35]{b:Handbook}. 
	Let $G_{B_{1}}$ contain $[q^{n-1}]:\SL_{n-1}(q)$. Then $[q^{n-1}]:\SL_{n-1}(q)\leq G_{B_{1}}\leq P$ and $b=b_{1}\cdot(q^{n}-1)/(q-1)$, where $b_{1}\mid (q-1)$. If $\SL_{n-1}(q)\leq G_{\alpha_{1},B_{1}}$, then $[q^{n-1}]:\SL_{n-1}(q)\leq G_{\alpha_{1},B_{1}}$, and so $k$ divides $q-1$, which is a contradiction as it violates $bk = vr$ for $n>3$. Thus, $|\alpha_{1}^{\SL_{n-1}(q)}|>1$, and hence $k\geq (q^{n-1}-1)/(q-1)$ implying that $b_{1}=1$, that is to say, $v=b$. Therefore, $\Dmc$ is symmetric, and hence Proposition~\ref{prop:psl} implies that $\Dmc=\PG_{n-1}(q)$, which is part (a).
\end{proof}

\begin{proposition}\label{prop:psp}
	Let $\Dmc$ be a nontrivial $2$-design with $\gcd(r, \lambda)=1$. Suppose that $G$ is an automorphism group of $\Dmc$ of almost simple type with socle $X$. If $G$ is flag-transitive, then the socle $X$ cannot be $\PSp_{2m}(q)$ with $m \geq  2$ and $(m,q) \neq (2,2)$.
\end{proposition}
\begin{proof}
	Let $H_{0}=H \cap X$, where $H=G_{\alpha}$ for some point $\alpha$ of $\Dmc$.
	Since $H$ is maximal in $G$, by Aschbacher's Theorem~\cite{a:Aschbacher-84}, $H \in \Cmc_i \cup \Smc$ ($1\leq i\leq8$).
	By Corollary~\ref{cor:large-2}, we have
	that $|G|<|H|{\cdot}|H|^2_{p'}$ except for parabolic subgroups, and by the same argument as in the proof of Proposition~\ref{prop:psl} and using Lemma~\ref{lem:coprime} and~\cite[Theorem 7 and Proposition 4.22]{a:AB-Large-15}, we have one of the following possibilities for the subgroup $H$:\smallskip
	
	\begin{enumerate}[\rm (1)]
		\item $H \in \Cmc_{1}\cup\Cmc_{6}\cup \Cmc_{8}\cup \Smc$;
		\item $H$ is a $\Cmc_2$-subgroup of type $\Sp_{2m/t}(q)\wr \S_t$ with $t\in \{2, 3, 4, 5\}$;
		\item $H$ is a $\Cmc_2$-subgroup of type $\GL_{m}(q)$;
		\item $H$ is a $\Cmc_3$-subgroup of type $\Sp_{2m/t}(q^t)$ with $t=2, 3$ or $\GU_{m}(q)$;
		\item $H$ is a $\Cmc_5$-subgroup of type $\Sp_{2m}(q_{0})$ with $q=q_{0}^2$.
	\end{enumerate}
	We now discuss each of these possible cases separately. \smallskip
	
	\noindent\textbf{(1)}  Let $H$ be in $\Cmc_{1}$. Then $H$ is reducible, and  it is either parabolic, or stabilizer of a nonsingular subspace.
	
	Suppose first  that $H=P_{i}$, the stabilizer of a totally singular $i$-subspace of $V$, with $i \leq  m$. Then by \cite[Lemma 5]{a:Korableva-Sp}, we have that
	\begin{align*}
		v=\dfrac{(q^{2m}-1)(q^{2m-2}-1)\cdots (q^{2m-2i+2}-1)}{(q^i-1)(q^{i-1}-1) \cdots (q-1)}.
	\end{align*}
	Thus $v \equiv q+1 \pmod {pq}$ and $q$ is the highest power of $p$ dividing $v-1$. By Lemma~\ref{lem:subdeg}, there is a subdegree which is a power of $p$, and Lemma~\ref{lem:six}(d) implies that
	$r$ divides $q$. Since $m \geq  2$, we have that $q^2+q+1\leq v$ and so
	$q^2+q+1\leq v<r^2 \leq  q^2$, which is a contradiction.
	
	Suppose now  that $H=N_{2i}$, the stabilizer of a nonsingular $2i$-subspace $U$ of $V$, with $2i<m$.
	By~\cite[Proposition 4.1.3]{b:KL-90},
	$H_{0} {\cong}\,^{\hat{}}\Sp_{2i}(q){\times}\Sp_{2(m-i)}(q)$.
	It follows from~\cite[Lemma 4.2 and Corollary 4.3]{a:AB-Large-15} and~\eqref{eq:v} that $v>q^{4i(m-i)}$.
	By  Lemma~\ref{lem:Tits}, $p$ divides $v$, so $\gcd(p,v-1)=1$. Thus, according to Lemma~\ref{lem:six}(a), the parameter $r$ is coprime to $p$.
	Here by \cite[p. 327]{a:Saxl2002}, there is a $H$-orbit with the $p'$-part of its length dividing $(q^{2i}-1)(q^{2m-2i}-1)/(q^2-1)^{2}$ which is divisible by $r$.
	Therefore, $r<4q^{2m-4}$. We now apply Lemma~\ref{lem:six}(c) and conclude that $q^{4i(m-i)}<v<r^2<16q^{4m-8}$.  Since $q\geq 2$, it follows that $4i(m-i)<4m-4$, and since $m>2i$, we have that  $i^{2}-2i+1<0$, which is impossible. \smallskip
	
	Let $H$ be a $\Cmc_6$-subgroup. Then by~\cite[Proposition 4.6.9]{b:KL-90} and the inequality $|G|<|H|^3$,
	we need only to consider the pairs $(X,H_0)$ listed in Table~\ref{tbl:psp-c6-s}. For
	each such $H_0$, by~\eqref{eq:v}, we obtain $v$ as in the fourth column of Table~\ref{tbl:psp-c6-s}. Moreover,
	Lemma~\ref{lem:six}(a)-(c) implies that $r$ divides $\gcd(|H|,v-1)$, and so we can find an upper bound
	$u_r$ of $r$ as in the fifth column of Table~\ref{tbl:psp-c6-s}. Then the inequality $\lambda v<r^2$ rules out
	all these possibilities.
	
	Let now $H$ be a $\Cmc_{8}$-subgroup. Then by~\cite[Proposition 4.8.6]{b:KL-90}, we have that $H_{0}\cong\O_{2m}^{\e}(q)$ with $q$ even.
	In this case from~\eqref{eq:v},  $v=q^{m}(q^{m}+\e)/2$. It follows from Lemmas~\ref{lem:subdeg-PSp} and ~\ref{lem:six}(d) that
	the parameter $r$ divides $(q^m- \e) \cdot \gcd(q-2, q^{m-1}+ \e)$, and hence Lemma \ref{lem:divisible} implies that $r$ is divisible by the index of a parabolic subgroup in $\O_{2m}^{\e}(q)$ and it is clearly impossible.\smallskip
	
	Let finally $H$ be a $\Smc$-subgroup. If $m\leq 6$, then the subgroups $H$ are listed in~\cite[Chapter 8]{b:BHR-Max-Low}. Since $|G|<|H|{\cdot}|H|^2_{p'}$,
	we only have to consider the pairs $(X,H_0)$ listed in Table~\ref{tbl:psp-c6-s}. For each such $H_0$ except for ${}^2\!\B_{2}(q)$ and $\G_2(q)$, we have $\lambda v>r^2$, which is a contradiction. For the remaining two cases, we note that $r$ is coprime to $p$ by Lemma~\ref{lem:Tits}.
	If $H_0\cong {}^2\!\B_{2}(q)$, then $v=q^2(q^2-1)(q+1)$. Note that the index of the parabolic subgroup in ${}^2\!\B_{2}(q)$ is $q^2+1$, and so  Lemma~\ref{lem:divisible} implies that $r$ is divisible by $q^2+1$. On the other hand, $r$ divides $v-1$.
	Therefore, $q^2+1$ must divide $v-1$, and so $q=2$, which is a contradiction.
	If $H_0\cong \G_2(q)$, then $v=q^3(q^4-1)$. Since $\gcd(r,v)=1$, $r$ is coprime to $q+1$.
	On the other hand, by Lemma~\ref{lem:divisible},  $(q^6-1)/(q-1)$ divides $r$, which is a contradiction.
	Therefore, $m>6$, and then \cite[Theorem 4.2]{a:Liebeck1985} implies that $|H|<q^{4m+4}$, $H'=\A_{2m+1}$ or $\A_{2m+2}$, or $(X,H_0)$ is $(\PSp_{56}(q),\E_7(q))$ with $q$ odd. If $H_{0}\cong \E_{7}(q)$, then clearly $v>r^{2}$, which is a contradiction.
	If the former case holds, then by Corollary~\ref{cor:large} and~\cite[Corollary 4.3]{a:AB-Large-15}, we have that $q^{m(2m+1)}/4<|G|<|H|^3<q^{12(m+1)}$, and so  $q^{m(2m+1)}<4q^{12(m+1)}$. Thus $2m^{2}-11m-14<0$, and this inequality has no solution for $m>6$.
	If $H'=\A_{2m+1}$ or $\A_{2m+2}$, then $H=\S_{2m+2}$ is contained in $G=\Sp_{2m}(2)$ with $m$ even.
	Thus the inequality $2^{m(2m+1)}/4<|G|<|H|^3\leq[(2m+2)!]^3$ yields $m=2$, $4$, $6$, $8$ or $10$, and hence by the bounds of $v$ and $r$ given in Table~\ref{tbl:psp-c6-s}, $r$ is too small to satisfy $\lambda v<r^{2}$.\smallskip
	
	\begin{table}
		\caption{Some maximal subgroups of $X=\PSp_{2m}(q)$.}\label{tbl:psp-c6-s}
		\resizebox{\textwidth}{!}{
			\begin{tabular}{clllll}
				\noalign{\smallskip}\hline\noalign{\smallskip}
				Class & $X$ & $H_0$ & $l_v$ & $u_r$ & Conditions \\
				\noalign{\smallskip}\hline\noalign{\smallskip}
				$\Cmc_{6}$ &
				$\PSp_{4}(3)$ &
				$2^4.\rm{\Omega}_{4}^{-}(2)$ &
				$3^{3}$ &
				$2$
				\\
				$\Cmc_{6}$ &
				$\PSp_{4}(5)$&
				$2^{4}.\rm{\Omega}_{4}^{-}(2)$ &
				$3{\cdot} 5^{3}{\cdot} 13$ &
				$2$
				\\
				$\Cmc_{6}$ &
				$\PSp_{4}(7)$ &
				$ 2^4.\rm{\Omega}_{4}^{-}(2)$ &
				$2^{2}{\cdot} 3{\cdot} 5{\cdot} 7^{4}$ &
				$1$
				\\
				&
				&
				$ 2^5.\rm{\Omega}_{4}^{-}(2)$ &
				$2{\cdot} 3{\cdot} 5{\cdot} 7^{4}$ &
				$1$
				\\
				$\Cmc_{6}$ &
				$\PSp_{4}(11)$ &
				$ 2^4.\rm{\Omega}_{4}^{-}(2)$ &
				$3{\cdot}5{\cdot}11^4{\cdot}61$ &
				$2$
				\\
				$\Cmc_{6}$ &
				$\PSp_{8}(3)$&
				$2^{6}.\rm{\Omega}_{6}^{-}(2)$ &
				$2^{2}{\cdot} 3^{12}{\cdot} 5{\cdot} 7{\cdot} 13{\cdot} 41$ &
				$1$
				\\
				&
				&
				$2^{7}.\rm{\Omega}_{6}^{-}(2)$ &
				$2{\cdot} 3^{12}{\cdot} 5{\cdot} 7{\cdot} 13{\cdot} 41$ &
				$1$
				\\
				$\Smc$ &
				$\PSp_{4}(q)$ &
				${}^2\!\B_{2}(q)$ &
				$q^2(q^2-1)(q+1)$&
				-&
				$q=2^{a}\geq 4$
				\\
				$\Smc$ &
				$\PSp_{4}(q)$ &
				$\PSL_2(q)$ &
				$q^3(q^4-1)$&
				$a$ &
				$q\geq5$
				\\
				$\Smc$ &
				$\PSp_{4}(q)$ &
				$\A_{6}{\cdot}c$ &
				$q^4(q^2-1)(q^4-1)$ &
				$3^{2}{\cdot}{5}$ &
				$q=p\geq5$, $c\leq2$
				\\
				$\Smc$ &
				$\PSp_{6}(q)$ &
				$\G_{2}(q)$ &
				$q^3(q^4-1)$&
				- &
				$q$ even\\
				$\Smc$ &
				$\PSp_{6}(5)$&
				$\mathrm{J}_{2}$  &
				$2^{2}{\cdot} 3{\cdot} 5^{7}{\cdot} 13{\cdot} 31$ &
				$1$&
				\\
				$\Smc$ &
				$\PSp_{8}(2)$ &
				$\S_{10}$ &
				$2^{8}{\cdot} 3{\cdot} 17$ &
				$5{\cdot}7$ &
				\\
				$\Smc$ &
				$\PSp_{12}(2)$ &
				$\S_{14}$ &
				$2^{25}{\cdot} 3^{3}{\cdot} 5{\cdot} 17 {\cdot} 31$ &
				$7^2$ &
				\\
				$\Smc$ &
				$\PSp_{16}(2)$ &
				$\S_{18}$ &
				$2^{48}{\cdot} 3^{2}{\cdot} 5{\cdot} 17{\cdot} 31{\cdot} 43{\cdot} 127{\cdot} 257$ &
				$1$ &
				\\
				$\Smc$ &
				$\PSp_{20}(2)$ &
				$\S_{22}$ &
				$2^{81}{\cdot} 3^{5}{\cdot} 5^{2}{\cdot} 17{\cdot} 31^{2}{\cdot} 41{\cdot} 43{\cdot} 73{\cdot} 127{\cdot} 257$ &
				$1$ &
				\\
				\noalign{\smallskip}\hline\noalign{\smallskip}
			\end{tabular}
		}
	\end{table}
	
	\noindent\textbf{(2)} Let $H$ be  a $\Cmc_2$-subgroup of type $\Sp_{2m/t}(q)\wr\S_t$ with $t\in \{2, 3, 4, 5\}$. Then by~\cite[Proposition 4.2.10]{b:KL-90}, we have that $H_{0}\cong \,^{\hat{}}\Sp_{2i}(q)\wr \S_t$, with $it=m$.
	By Lemmas~\ref{lem:subdeg-PSp} and~\ref{lem:six}(d), we conclude that $r$ divides
	\begin{align}\label{psp(2)-subdeg}
		\frac{ t(t-1)(q^{2i}-1)^2}{2(q-1)}.
	\end{align}
	It follows from~\cite[Lemma 4.2 and Corollary 4.3]{a:AB-Large-15} and~\eqref{eq:v} that $v>q^{2i^2t(t-1)}/(t!)$. Since $v<r^2$, we have that
	$4(q-1)^2q^{2i^2t(t-1)}<(t!)v<(t!)r^2<(t!)t^2(t-1)^2(q^{2i}-1)^4$. So $q^{2i^2t(t-1)-8i+2}<(t!)t^4$, and then $q^{2t(t-1)-6}<t^{t+4}$. This inequality holds only when $t\in \{2,3\}$.
	
	Let first $t = 3$. Then by~\eqref{psp(2)-subdeg}, $r<6q^{4i-1}$. We apply Lemma~\ref{lem:six} and deduce that $q^{12i^2-8i+2}<6^3$, and so $q=2$ and $i=1$. Therefore, $X=\PSp_{6}(2)$ and $v=1120$; but then according to \cite[p. 46]{b:Atlas} $H$ is not maximal, which is a contradiction.
	
	Let now $t=2$. Then again by~\eqref{psp(2)-subdeg}, $r<2q^{4i-1}$, so Lemma \ref{lem:six}(c) implies that $q^{4i^2-8i+2}<8$. Thus $i=1,2$. If $i=2$, then the last inequality yields $q=2$, and hence $v=45696$.
	By~\eqref{psp(2)-subdeg} and Lemma~\ref{lem:six}(a) and (d), we have that $r\leq \gcd(45695,255)=5$, and this violates Lemma~\ref{lem:six}(c). Therefore, $i=1$, $X= \PSp_4(q)$ and $v=q^2(q^2+1)/2$.
	The subdegrees of X are shown in Table~\ref{tbl:subdeg-PSp}, and they are
	\begin{align*}
		&(q^2-1)(q+1),q(q^2-1)/2,\text{ and } q(q^2-1)(q-3)/2 \text{ if } q \text{ is }\text{odd }, \text{ and } \\
		&(q^2-1)(q+1), \text{ and } q(q^2-1)(q-2)/2  \text{ if } q  \text{ is } \text {even.}
	\end{align*}
	By~\eqref{psp(2)-subdeg} and the fact that $r$ divides each of these subdegrees, we conclude that $r$ divides $(q^2-1)(q+1)$. Now by Lemma~\ref{lem:six}(a), $r$ is a divisor of $(v-1)$, and so $r$ is a divisor of $\gcd ((q^2-1)(q+1), (q^2-1)(q^2+2)/2)$ dividing $3(q^2-1)$. Therefore, $r=3(q^2-1)/s$, for some  positive integer $s$.  Since $\lambda \geq  2$ and since $\lambda v<r^2$, we have that $r>(q^2-1)$, and hence $s=1,2$.
	Assume first that $s=1$. Then $r=3(q^2-1)$. If $k=r$, then $k=3(q^2-1)$, and so the condition $k(k-1)=\lambda(v-1)$ implies that $\lambda=6(3q^2-4)/(q^2+2)$. Thus $q^2+2$ divides $6(3q^2-4)=18(q^{2}+2)-60$ implying that $q^{2}+2$ must divide $60$. This is true for $q=2$ in which case we obtain the trivial design with parameters $(10,9,8)$. If $r>k$, then $k= 1+\lambda(q^2+2)/6$ and $b=vr/k=3q^2(q^4-1)/2k$.
	Now by Lemma~\ref{lem:six}(c), we have that
	\begin{align*}
		\lambda(\frac{q^2(q^2+1)}{2}) < 9(q^2-1)^2,
	\end{align*}
	and so $\lambda < {18(q^2-1)^2/q^2(q^2+1)}< 18$. It follows that $\lambda =2, \ldots ,17$.
	Then by calculation, we obtain the  parameters set $(v,b,r,k, \lambda)$ is   $(45,90,24,12,6)$,
	$(136,612,45,10,3)$,
	$(325,2340,72,10,2)$,
	$(1225,9800,144,18,2)$,
	$( 1225 , 5040 , 144 , 35 , 4 )$ or
	$( 2080, 8736, 189, 45, 4 )$
	when  $q$ is $3$, $4$, $5$, $7$, $7$ or $8$, respectively. Since $r$ has to be coprime to $\lambda$, all these cases can be ruled out except for $(2080,8736,189,45,4)$ when $q=8$. In this case, $X=\PSp_4(8)$, and so $H$ is isomorphic to $\PSL_2(8)^2{:}2$ or $\PSL_2(8)^2{:}6$, but using \textsf{GAP} \cite{GAP4}, neither of these subgroups has a subgroup of index $189$, which is a contradiction. Assume now that  $s=2$. Then $r=3/2(q^2-1)$ and again by Lemma~\ref{lem:six}(c), we have that
	\begin{align*}
		\lambda(\frac{q^2(q^2+1)}{2}) < \frac{9}{4}(q^2-1)^2.
	\end{align*}
	This yields ${\lambda < 9(q^2-1)^2/2 q^2(q^2+1)<9/2}$. Hence $\lambda =2,3,4$. Thus we only obtain  a symmetric $(45,12,3)$ design for $q=3$, but here $\gcd(r, \lambda)=3$, which is a contradiction.\smallskip
	
	\noindent\textbf{(3)}  Let $H$ be  a $\Cmc_2$-subgroup of type $\GL_{m}(q)$. Then by~\cite[Proposition 4.2.5]{b:KL-90}, we have that $H_{0}\cong\,^{\hat{}}\GL_{m}(q){\cdot} 2$  and $q$ is odd. Here by~\eqref{eq:v}, we have that
	\begin{align*}
		v= q^{m(m+1)/2}(q^{m}+1)(q^{m-1}+1)\cdots (q+1)/2>q^{m(m+1)}/2.
	\end{align*}
	Note that $r$ is coprime to $p$. We observe by~\cite[p. 327]{a:Saxl2002} that there is a $H$-orbit with the $p'$-part of its length dividing
	$2(q^m-1)$, and so Lemma~\ref{lem:six}(d) implies that $r$ divides $2(q^m-1)$. By Lemma~\ref{lem:six}(c), we must have $q^{m(m+1)}<2v<2r^2<8(q^m-1)^2$. Therefore, $q^{m^2-m}<8$, which is impossible as $m \geq  2$ and $q$ is odd. \smallskip
	
	\noindent\textbf{(4)} Let $H$ be a $\Cmc_3$-subgroup of type $\Sp_{m}(q^2)$, $\Sp_{2m/3}(q^3)$ or $\GU_{m}(q)$. Then by~\cite[Proposition 4.3.7 and 4.3.10]{b:KL-90}, $H_{0}$ is isomorphic to one of the following subgroups:
	\begin{enumerate}[(a)]
		\item $^{\hat{}}\GU_{m}(q){\cdot} 2$ with $q$ odd;
		\item $\PSp_{2i}(q^t){\cdot} t$ with $m=it$ and $t=2,3$.
	\end{enumerate}
	
	Assume first that $H_{0}\cong \,^{\hat{}}\GU_{m}(q){\cdot} 2$ with $q$ odd. Note that $v$ is even. Then by Lemma~\ref{lem:six}(a), $r$ must be odd. Also  Lemma~\ref{lem:Tits} says that  $r$ is coprime to $p$. Then by Lemma~\ref{lem:divisible}, the stabiliser of a block under $H_{0}$ is contained in a parabolic subgroup of $\GU_{m}(q)$, and so Lemma~\ref{lem:odd-deg} implies that $r$ is even. But here $v-1$ is odd, which is a contradiction.\smallskip
	
	Assume now that $H_{0}\cong\PSp_{2i}(q^t){\cdot} t$ with $m=it$ and $t=2, 3$.
	By \cite[Lemma 4.2]{a:AB-Large-15} and \eqref{eq:v}, we have that $v>q^{2i^{2}t(t-1)}/4t$, where $t=2, 3$.
	Applying Corollary~\ref{cor:large-2} and \cite[Lemma 4.2]{a:AB-Large-15}, and since $a^2\leq2q$, we have that $|\Out(X)|^2\leq8q$. Thus the inequality $|X|< |\Out(X)|^2{\cdot} |H_{0}|{\cdot} |H_{0}|_{p'}^{2}$ yields
	\begin{align*}
		\frac{1}{4} q^{2i^2t^2+it}<8t^3 \cdot q^{2i^2t+it+1} \cdot \prod_{j=1}^{i}(q^{2tj}-1)^2.
	\end{align*}
	If $t=3$, then Lemma~\ref{lem:equ}(a) implies that $q^{6i(i-1)-1}<864$. This inequality holds only for $i=1$.
	So $X=\PSp_{6}(q)$ and $H_0=\PSp_{2}(q^3){\cdot}3$ . By \eqref{eq:v}, we have that $v=q^6(q^4-1)(q^2-1)/3$. Since $q+1$ divides $v$, it follows from Lemma~\ref{lem:six}(a) that $r$ is coprime to $q+1$. On the other hand, applying Lemma~\ref{lem:divisible} to $\PSp_{2}(q^3)$, we deduce that $r$ is divisible by the index of a parabolic subgroup in $\PSp_{2}(q^3)$. Thus $q^3+1$ divides $r$, which is a contradiction.
	
	If $t=2$, then $v>q^{4i^2}/8$. By Lemmas~\ref{lem:subdeg-PSp} and~\ref{lem:six}(d), $r$ divides $a(q^{4i}-1)$. It follows from Lemma~\ref{lem:six}(c) and $a^2\leq 2q$ that $q^{4i^2}/8<v<r^{2}<2q^{8i+1}$, that is to say, $4i^{2}-8i-5<0$. Thus $i \leq  2$.
	Moreover, if $i=2$, then $v=q^8(q^6-1)(q^2-1)/2$ and $r$ divides $a(q^{8}-1)$. Note by Lemma~\ref{lem:six}(a) that $r$ is coprime to $v$, and since $q^2-1$ divides $v$, we conclude that $r$ divides $a(q^4+1)(q^2+1)$. Thus $r<q^7$, and Lemma~\ref{lem:six}(c) forces
	$\lambda q^{16}/8<\lambda v<r^2<q^{14}$, and so $\lambda q^{2}<8$, but this inequality dose not hold for $\lambda \geq2$.
	
	Therefore, $i=1$. Thus $H_0=\PSp_{2}(q^2){\cdot}2$ and  $X=\PSp_{4}(q)$. Lemmas~\ref{lem:subdeg-PSp} and~\ref{lem:six}(d) imply that $r$ is a divisor of $q^2+1$. From this and the fact that  $v=q^2(q^2-1)/2<r^2$, we deduce that $s^{2}q^2(q^2-1)<2(q^2+1)^{2}$, where $rs=q^{2}+1$, for some positive integer $s$, but this is true only when $s=1$. Hence $r=q^{2}+1$ and $q$ is even. If $k=r=q^2+1$, then the equality $k(k-1)=\lambda(v-1)$ implies that $\lambda=2q^2/(q^2-2)$. But since $\gcd(q^2,q^2-2)=1$ or $2$, the parameter $\lambda$ has an integer value only when $q=2$ in which case  $\Dmc$ is a trivial design with parameters $(6,5,4)$. If $r>k$, then we apply Lemma \ref{lem:six}(a) and conclude that $k=1+\lambda(q^2-2)/2$, and since $k< r$, we must have $\lambda(q^2-2)< 2q^2$. By excluding the case where $X=\PSp_{4}(2)'\cong \A_{6}$, this inequality implies that $\lambda=1,2$, and by \cite{a:Saxl2002}, we can assume that $\lambda= 2$. In this case, by Lemma \ref{lem:six}(a)-(b), we have that $k=q^2-1$ and $b=q^2(q^2+1)/2$.
	Now by Lemma~\ref{lem:embed}, the design $\Dmc$ can be embedded into  a symmetric design with parameters $((q^4+q^2+2)/2,q^2+1,2)$ of order $q^{2}-1$ with $q=2^{a}$. But such a biplane does not exist for $a\geq 4$ by \cite[Section 2]{a:Assmus-1991}. If $a$ is $2$ or $3$, then the biplane has $137$ or $2081$ number of points, respectively, and so by \cite{a:Braic-2500-nopower}, we cannot find any biplanes with these number of points.\smallskip
	
	
	\noindent\textbf{(5)} Let $H$ be  a $\Cmc_5$-subgroup of type $\Sp_{2m}(q_{0})$ with $q=q_{0}^2$. In this case, by~\cite[Proposition 4.5.4]{b:KL-90}, we have that $H_{0}\cong \PSp_{2m}(q_{0}){\cdot} c$ with $q=q_{0}^2$ and $c\leq 2$, (with $c=2$ if and only if $q$ is odd).
	By \eqref{eq:v}, we observe that $v>q_0^{m(2m+1)}/2$. It follows from \cite[p. 329]{a:Saxl2002} that there is a subdegree of $X$ with the $p'$-part dividing $q_0^{2m}-1$, and so $r$ divides $a(q_0^{2m}-1)$. Applying Lemma~\ref{lem:six}(c), we conclude that $q_0^{m(2m+1)}<2v<2a^2(q_0^{2m}-1)^2$, and hence $q_0^{m(2m+1)}<2a^2(q_0^{2m}-1)^2$. Thus $q_0^{2m^2-3m}<2a^{2}$. This inequality holds only for $m=2$ and $q_{0}=2, 4, 8$, and so $v$ is $1360$, $1118464$ or $1090785280$, respectively. Therefore,  $r$ divides $9$, and hence $r$ is too small satisfying $\lambda v<r^2$, which is a contradiction.
\end{proof}

\begin{proposition}\label{prop:orth}
	Let $\Dmc$ be a nontrivial $2$-design with $\gcd(r, \lambda)=1$. Suppose that $G$ is an automorphism group of $\Dmc$ of almost simple type with socle $X$. If $G$ is flag-transitive, then the socle $X$ cannot be $\POm_{n}^{\e}(q)$ with $\e \in \{\circ, -, +\}$.
\end{proposition}
\begin{proof}
	Let $H_{0}=H\cap X$, where $H=G_{\alpha}$ for some point $\alpha$ of $\Dmc$. Since $H$ is maximal in $G$, by Aschbacher's Theorem~\cite{a:Aschbacher-84}, $H \in \Cmc_i \cup \Smc$ ($1\leq i\leq7$).
	By the same argument as in the proof of Proposition \ref{prop:psl} and using Lemma \ref{lem:coprime} and \cite[Theorem~7 and Proposition~4.23]{a:AB-Large-15}, we have one of the following possibilities for $H$:
	\begin{enumerate}[\rm (1)]
		\item $H \in \Cmc_{1} \cup \Cmc_6 \cup \Smc$;
		\item $H$ is a $\Cmc_2$-subgroup of type $\O_{n/t}^{\e'}(q)\wr\S_t$ and one of the following holds:
		\begin{enumerate}[(a)]
			\item $t=2$;
			\item $(n, t, q, \e, \e')=(12, 3, 2, -, -), (10, 5, 2, -, -)$ or $(8, 4, 2, +, -)$;
			\item $n=t$;
		\end{enumerate}
		\item $H$ is a $\Cmc_2$-subgroup of type $\GL_{n/2}(q)$;
		\item $H$ is a $\Cmc_2$-subgroup of type $\GO^{\e'}_2(q) \wr\S_{n/2}$ with $\e=(\e')^{n/2}$; 
		\item $H$ is a $\Cmc_3$-subgroup of type $\O_{n/2}^{\e'}(q^2)$ or $\GU_{n/2}(q)$;
		\item $H$ is a $\Cmc_4$-subgroup of type $\Sp_{n/2}(q)\otimes \Sp_{2}(q)$ and $(n, \e) \in\{(12, +), (8, +)\}$;
		\item $H$ is a $\Cmc_5$-subgroup of type $\O_{n}^{\e'}(q_{0})$ with $q=q_{0}^2$.
	\end{enumerate}
	
	We analyse each of these possible cases separately and arrive at a contradiction in
	each case. Note that we postpone the case where $(m,\e)=(4,+)$ and $G$ contains a triality automorphism until the end of the proof.\smallskip
	
	\noindent\textbf{(1)} Let $H$ be in $\Cmc_{1}$. Then $H$ stabilises a totally singular $i$-subspace with $2i\leq n$ or a nonsingular subspace.
	
	Assume first that $H$ stabilises a totally singular $i$-subspace. If $n$ is odd, we argue exactly the same as in the symplectic case in Proposition \ref{prop:psp} and obtain no possible parameters.
	Let now $n=2m$ be even, and suppose that $i<m$. Then $H=P_{i}$ unless $i=m-1$ and $\e=+$, in this case, $H=P_{m, m-1}$. Note by Lemma~\ref{lem:subdeg} that there is a subdegree which is a power of $p$ except for the case where $\e=+$, $n/2$ is odd and $H=P_{m}$ or $H=P_{m-1}$. On the other hand, the highest power $(v-1)_{p}$ of $p$ dividing $v-1$ divides $q^2$ or $8$, so $r$ is too small to satisfy $\lambda v<r^2$.
	
	Assume now that $H=P_{m}$ when $X=\POm_{2m}^{+}(q)$. Note here that $P_{m}$ and $P_{m-1}$ are the stabilisers of totally singular $m$-spaces from the two different $X$-orbits.
	If $m$ is even and $m>4$, we conclude by Lemma~\ref{lem:subdeg} that $G$ has a subdegree of power of $p$. This case again can be ruled out as $(v-1)_{p}=q$ and $r$ is too small to satisfy $\lambda v<r^2$. This leaves the case where $m\geq 5$ is odd.  Then by \cite[Proposition 4.1.20]{b:KL-90} and \eqref{eq:v}, we have that
	$v=(q^{m-1}+1)(q^{m-2}+1)\cdots(q+1)$, and so $v>q^{m(m-1)/2}$.
	In this case, by~\cite[p. 332]{a:Saxl2002}, there is a subdegree of $G$ with the $p'$-part dividing $q^{m}-1$. We note that $(v-1)_{p}=q$. Thus by Lemma~\ref{lem:six}(d), we deduce that $r\leq q(q^{m}-1)$, and so Lemma \ref{lem:six}(c) implies that $q^{m(m-1)/2}< v<r^2\leq q^2(q^{m}-1)^2<q^{2m+2}$. This inequality yields $q^{m^2-m}<q^{4m+4}$, that is to say, $m^2-5m-4<0$, and so $m=5$. We now apply Lemma~\ref{lem:subdeg-PSp} and conclude that $r$ divides $q(q^5-1)/(q-1)$. Thus $r=q(q^5-1)/s(q-1)$, for some positive integer $s$. We then observe that the inequality $\lambda q^{10}\leq \lambda v<r^2\leq q^2(q^5-1)^2/s^2(q-1)^2$ yields $\lambda s^2 q^{10}(q-1)^2<q^2(q^5-1)^2$, and this holds only for $s=1$ and $\lambda \in\{2,3\}$, and so  $r=q(q^5-1)/(q-1)$. Then by Lemma \ref{lem:six}, we have that $(r,\lambda,q) \in \{(62,2,2),(363,2,3),(62,3,2)\}$, but none of these cases give rise to any possible parameters.
	
	Assume finally that $H$ is the stabiliser of a nonsingular $i$-subspace, that is to say, $H=N_i^{\delta}$ with $i \leq  m$. Let $n=2m+1$, $q$ odd and $\e=\circ$. If $i=1$, then by \eqref{eq:v}, $v=q^{m}(q^{m}+\delta)/2$. We apply Lemmas~\ref{lem:subdeg-PSp} and~\ref{lem:six}(d) and conclude that $r\leq (q^{m}-\delta)/2$. So Lemma~\ref{lem:six}(c) implies that $2q^{m}(q^{m}+\delta)<(q^{m}-\delta)^2$, which is impossible. Thus $i\geq2$. By \cite[Proposition 4.1.6]{b:KL-90} and \eqref{eq:v}, we have that $v>q^{i(n-i)}/4$, and by \cite[p. 331]{a:Saxl2002}, $r\leq 2aq^m$. It follows from Lemma~\ref{lem:six}(c) that $q^{4m-2}/4\leq q^{i(2m+1-i)}/4<v<r^2\leq4a^2q^{2m}$, and so  $q^{2m-2}<16a^2$,  which is impossible as $m\geq 3$. Let now $n=2m$ and $\e=\pm$. If $H=N_i$ with $i=1$, then $v=q^{m-1}(q^{m}-\e)/\gcd(2, q-1)$, and by Lemma~\ref{lem:subdeg-PSp}, we conclude that $r\leq  (q^{m-1}+1)/\gcd(2, q-1)$. So Lemma~\ref{lem:six}(c) implies that  $q^{m-1}(q^{m}-\e)<2(q^{m-1}+1)^2$ if $q$ is odd, and $q^{m-1}(q^{m}-\e)<(q^{m-1}+1)^2$ if $q$ is even, but both cases are impossible. Hence $H=N_i^\delta$, with $1<i \leq m$. Note that $\delta=\pm$ present only if $i$ is even. Also $v>q^{i(n-i)}/4$. If $q$ is odd, then by \cite[p. 333]{a:Saxl2002},  $r<4aq^m$. Here by Lemma \ref{lem:six}(c), we have that $q^{2m-4}<64a^2$, which is impossible for $m\geq4$. Thus $q$ is even, and hence $i$ is also even. Assume first that ${i=2}$. It follows from Lemmas~\ref{lem:subdeg-PSp} and~\ref{lem:six}(d) that $r$ divides $a_{2'}(q-\delta)(q^{m-1}-\e \delta)$, and so $r<4a_{2'}q^m$. Thus Lemma~\ref{lem:six}(c) implies that $q^{4m-4}/4<v<r^2<16a^2q^{2m}$, whence $q^{2m-4}<64a^2$. This inequality holds only for $m=4$ and $q=2$. In this case, $r\leq a_{2'}(q+1)(q^3+1)$, and so $2^{10}=q^{12}/4<v<r^2\leq27^2$, which is impossible.
	Assume finally that $2<i\leq m$. Again by Lemmas~\ref{lem:subdeg-PSp} and~\ref{lem:six}(d), $r$ divides $a_{2'}(q^{i/2}-\delta)(q^{(i-2)/2}+\delta)(q^{(n-i)/2}+\delta')(q^{(n-i-2)/2}+\delta')$ with $\delta'=-\e \delta$, and so $r<8a_{2'}q^{2(m-1)}$. If $a$ is even, then Lemma~\ref{lem:six}(c) implies that $q^{8m-16}/4\leq q^{i(2m-i)}/4<v<r^2<16a^2q^{4m-4}$, and so  $q^{4m-12}<64a^2$. If $a$ is odd, then by the same argument, $q^{4m-12}<256a^2$. Both inequalities holds only for $m=4$ and $q=2$. This forces $i=4$, and so  $r\leq a_{2'}(q+1)^2(q^4-1)$ implies that $2^{14}<v<r^2\leq3^2{\cdot}15^2$, which is impossible.\smallskip
	
	Let now $H$ be a $\Cmc_6$-subgroup. Then by~\cite[Proposition 4.6.8]{b:KL-90} and the inequality $|G|<|H|^3$,
	we only need to consider the case where  $(X,H_0)=(\POm_8^+(3),2^6{\cdot}\A_8)$ and this case can be ruled out as  $3^{10}{\cdot}5{\cdot}13=v<r^2\leq2^6{\cdot}7^{2}$.\smallskip
	
	Let finally $H$ be a $\Smc$-subgroup. Then for $n\leq12$, the subgroups $H$ are listed in~\cite[Chapter 8]{b:BHR-Max-Low}.
	Since $|G|<|H|{\cdot}|H|^2_{p'}$, we only need to consider the pairs $(X,H_0)$ listed in Table~\ref{tbl:orth-c6-s}. If $H_0= \G_2(q)$ with $q$ odd, then $v=q^3(q^4-1)/2$. Since $\gcd(r,v)=1$, $r$ is odd. Then we apply Lemma~\ref{lem:divisible} to $\G_2(q)$ and by~\cite[Theorem~A]{a:Kleidman-G2}, we deduce that $(q^6-1)/(q-1)$ divides $r$, which is a contradiction.
	For the case where $(X,H_0)=(\POm_8^{+}(q), \rm{\Omega}_7(q))$, according to \cite[p. 335]{a:Saxl2002}, $r\leq(q^3+1)/2$, and so $v>r^2$, which is a contradiction. For all other cases, we use the fact that $r$ divides $\gcd(|H|,v-1)$ which similarly forces $\lambda v>r^2$, which is a contradiction. Hence $n\geq 13$.
	Then \cite[Theorem~4.2]{a:Liebeck1985} implies that (i) $|H|<q^{2n+4}$, (ii) $H'=\A_{n+1}$ or $\A_{n+2}$, or (iii) $X$ and $H_0$ are as in \cite[Table 4]{a:Liebeck1985}. If (i) holds, then by Corollary~\ref{cor:large-2} and \cite[Corollary 4.3]{a:AB-Large-15}, we have that $q^{n(n-1)/2}/8<|G|<|H|^3<q^{6n+12}$, and this yields, $q^{n^2-13n-24}<64$, that is to say, $n^2-13n-30<0$, whence $n=13,14$. If $n=13$, then  $|H|<q^{2n+4}=q^{30}$, and so $q^{78}/4<|G|<q^{30}|H|_{p'}^2$, and this yields $|H|_{p'}>q^{24}/2$. By the same method as \cite{a:Liebeck1985,a:Liebeck-HA-rank3} we have no possible choices for $H$.
	If $n=14$, then $|H|<q^{2n+4}=q^{32}$, and so it follows from Corollary~\ref{cor:large-2} and~\cite[Corollary~4.3]{a:AB-Large-15} that $q^{91}/8<|G|<|H|{\cdot}|H|^2_{p'}<q^{32}{\cdot}|H_{p'}|^2$, and hence $(2\sqrt{2})^{-1}q^{29}<|H|_{p'}$. This case also can be ruled out by the same argument as in~\cite[Sections 2, 3 and 5]{a:Liebeck-HA-rank3}. If (ii) or (iii) holds,  the fact that $|G|<|H|^3$ implies that $(X,H_{0})$ is  $(\rm{\Omega}_{14}^+(2),\A_{16})$, $(\rm{\Omega}_{16}^+(2),\A_{17})$, $(\rm{\Omega}_{18}^{-}(2),\A_{20})$, $(\rm{\Omega}_{20}^{-}(2),\A_{21})$, $(\rm{\Omega}_{22}^{+}(2), \A_{24})$ or $(\POm_8^{+}(q),\PSp_6(q))$  but all these cases can easily be excluded. Note that in the latter case, by Lemma~\ref{lem:divisible}, $q+1$ divides both $v$ and $r$, which is a contradiction by the fact that $\gcd(r,v)=1$.
	\smallskip
	
	\begin{table}
		\scriptsize
		\caption{\footnotesize Some maximal subgroups of $X=\POm^\e_n(q)$.}\label{tbl:orth-c6-s}
		\begin{tabular}{clllll}
			\noalign{\smallskip}\hline\noalign{\smallskip}
			Class & $X$ & $H_0$ & $l_v$ & $u_r$ & Conditions \\
			\noalign{\smallskip}\hline\noalign{\smallskip}
			$\Cmc_{2}$ &
			$\POm_{7}(3)$ &
			$2^{6}{\cdot}\A_{7}$  &
			$3^7{\cdot}13$ &
			$2{\cdot}5$ &
			\\
			$\Cmc_{2}$ &
			$\POm_{7}(5)$ &
			$2^{6}{\cdot}\A_{7}$  &
			$3^2{\cdot}5^{8}{\cdot}13{\cdot}31$ &
			$2$ &
			\\
			$\Cmc_{2}$ &
			$\POm_{9}(3)$ &
			$2^{8}{\cdot}\A_{9}$  &
			$3^{12}{\cdot}5{\cdot}13{\cdot}41$ &
			$2^3$ &
			\\
			$\Cmc_{2}$ &
			$\POm_{11}(3)$ &
			$2^{10}{\cdot}\A_{11}$  &
			$3^{21}{\cdot}11{\cdot}13{\cdot}41{\cdot}61$&
			$2^{4}{\cdot}7$ &
			\\
			$\Cmc_{2}$ &
			$\POm_{8}^{+}(3)$ &
			$2^{6}{\cdot}\A_{8}$  &
			$3^{10}{\cdot}5{\cdot}13$ &
			$2^3{\cdot}7$ &
			\\
			$\Cmc_{2}$ &
			$\POm_{10}^{-}(3)$ &
			$2^{8}{\cdot}\A_{10}$  &
			$3^{16}{\cdot}13{\cdot}41{\cdot}61$&
			$2^6$ &
			\\
			$\Cmc_{2}$ &
			$\POm_{8}^{+}(2)$ &
			$\rm{\Omega}_{2}^{-}(2)^4{\cdot} 2^{6}{\cdot}3$  &
			$2^{6}{\cdot}5^2{\cdot}7$ &
			$3$ &
			\\
			$\Cmc_{2}$ &
			$\POm_{8}^{+}(2)$ &
			$\rm{\Omega}_{2}^{+}(2)^4{\cdot} 2^{6}{\cdot}3$  &
			$2^{6}{\cdot}3^4{\cdot}5^2{\cdot}7$ &
			$1$ &
			\\
			$\Cmc_{2}$ &
			$\POm_{8}^{+}(3)$ &
			$\rm{\Omega}_{2}^{\pm}(3)^4{\cdot} 2^{9}{\cdot}3$  &
			$2^{3}{\cdot}3^{11}{\cdot}5^2{\cdot}7{\cdot}13$ &
			$1$ &
			\\
			$\Cmc_{2}$ &
			$\POm_{10}^{-}(2)$ &
			$\rm{\Omega}_{2}^{-}(2)^5{\cdot} 2^{7}{\cdot}3{\cdot}{5}$ &
			$2^{13}{\cdot}5{\cdot}7{\cdot}11{\cdot}17$ &
			$3$ &
			\\

			$\Cmc_{2}$ &
			$\POm_{10}^{+}(2)$ &
			$\rm{\Omega}_{2}^{+}(2)^5{\cdot} 2^{7}{\cdot}3{\cdot}{5}$ &
			$2^{13}{\cdot}3^4{\cdot}5{\cdot}7{\cdot}17{\cdot}31$ &
			$1$ &
			\\

			$\Cmc_{2}$ &
			$\POm_{12}^{-}(2)$ &
			$\rm{\Omega}_{4}^{-}(2)^3{\cdot} 2^{3}{\cdot}3$  &
			$2^{21}{\cdot}3^{2}{\cdot}7{\cdot}11{\cdot}13{\cdot}17{\cdot}31$ &
			$5$ &
			\\
			$\Smc$ &
			$\POm_{7}(q)$ &
			$\G_{2}(q)$ &
			$q^3(q^4-1)/2$ &
			-&
			$q$ odd
			\\
			$\Smc$ &
			$\POm_{7}(3)$ &
			$\PSp_{6}(2)$&
			$3^{5}{\cdot} 13$ &
			$2$ &
			\\
			$\Smc$ &
			$\POm_{7}(3)$ &
			$\S_{9}$  &
			$2^{2}{\cdot} 3^{5}{\cdot} 13$ &
			$5{\cdot}7$ &
			\\
			$\Smc$ &
			$\POm_{7}(5)$ &
			$\PSp_{6}(2)$ &
			$5^{8}{\cdot} 13{\cdot} 31$ &
			$2{\cdot}3$ &
			\\
			$\Smc$ &
			$\POm_{8}^{+}(q)$ &
			$\rm{\Omega}_{7}(q)$ &
			$q^3(q^4-1)/2$ &
			$(q^3+1)/2$ &
			$q$ odd\\
			$\Smc$ &
			$\POm_{8}^{+}(2)$ &
			$\A_{9}$ &
			$2^{6}{\cdot} 3{\cdot} 5$ &
			$7$ &
			\\
			$\Smc$ &
			$\POm_{8}^{+}(3)$ &
			$\rm{\Omega}_{8}^{+}(2)$ &
			$3^{7}{\cdot} 13$ &
			$2{\cdot}5$ &
			\\
			$\Smc$ &
			$\POm_{8}^{+}(5)$ &
			$\rm{\Omega}_{8}^{+}(2)$ &
			$5^{10}{\cdot} 13^{2}{\cdot} 31$ &
			$2{\cdot}3$ &
			\\
			$\Smc$ &
			$\POm_{8}^{+}(7)$ &
			$\rm{\Omega}_{8}^{+}(2)$ &
			$2^{4}{\cdot} 5^{2}{\cdot} 7^{11}{\cdot} 19{\cdot} 43$ &
			$3$ &
			\\
			$\Smc$ &
			$\POm_{10}^{-}(2)$ &
			$\A_{12}$ &
			$2^{11}{\cdot} 3{\cdot} 17$ &
			$7$ &
			\\
			$\Smc$ &
			$\POm_{10}^{+}(3)$ &
			$\A_{12}$ &
			$2^{6}{\cdot} 3^{15}{\cdot} 11{\cdot} 13{\cdot} 41$ &
			$1$ &
			\\
			$\Smc$ &
			$\POm_{12}^{-}(2)$ &
			$\A_{13}$ &
			$2^{21}{\cdot} 3{\cdot} 5{\cdot} 17{\cdot} 31$ &
			$1$ &
			\\
			$\Smc$ &
			$\POm_{14}^{+}(2)$ &
			$\A_{16}$ &
			$2^{28}{\cdot} 3^{2}{\cdot} 17{\cdot} 31{\cdot} 127$ &
			$5{\cdot}7$ &
			\\
			$\Smc$ &
			$\POm_{16}^{+}(2)$ &
			$\A_{17}$ &
			$2^{42}{\cdot} 3^{4}{\cdot} 5{\cdot} 17{\cdot} 31{\cdot} 43{\cdot} 127$ &
			$1$ &
			\\
			$\Smc$ &
			$\POm_{18}^{-}(2)$ &
			$\A_{20}$ &
			$2^{55}{\cdot} 3^{5}{\cdot} 17{\cdot} 31{\cdot} 43{\cdot} 127{\cdot} 257$ &
			$5{\cdot}13$ &
			\\
			$\Smc$ &
			$\POm_{20}^{-}(2)$ &
			$\A_{21}$ &
			$2^{73}{\cdot} 3^{4}{\cdot} 5^{2}{\cdot} 17{\cdot} 31{\cdot} 41{\cdot} 43{\cdot} 73$ &
			$13$ &
			\\
			$\Smc$ &
			$\POm_{22}^{+}(2)$ &
			$\A_{24}$ &
			$2^{89}{{\cdot}}3^4{{\cdot}}5^2{{\cdot}}17{{\cdot}}31^2{{\cdot}}41{{\cdot}}43{{\cdot}}73$ &
			$1$ &
			\\
			$\Smc$ &
			$\POm_{8}^{+}(q)$&
			$\PSp_{6}(q)$ &
			$q^3(q^4-1)$&
			-&
			$q$ even\\
			\noalign{\smallskip}\hline\noalign{\smallskip}
		\end{tabular}
	\end{table}
	\noindent\textbf{(2)} Let $H$ be a $\Cmc_2$-subgroup of type $\O_{m}^{\e'}(q)\wr\S_t$ with $m=n/t$. \smallskip
	
	\noindent{(a)}\quad Here $t=2$,  and so by \cite[Propositions 4.2.11 and 4.2.14]{b:KL-90}, the pair $(X,H_0)$ has to be $(\POm_{2m}^{+}(q), \rm{\Omega}_{m}^{\e'}(q)^2{\cdot} 2^{c})$ with $\e'=\pm$ and $c=2,3$, or $(\POm_{2m}^{\pm}(q),\rm{\Omega}_{m}(q)^2{\cdot}4)$ with $mq$ odd. In the former case, let first $m=4$.  If $\e'=+$, then $|H_0|=2^cq^4(q^2-1)^4/\gcd(4,q^{4}-1)$ with $c=2,3$, and so $v\geq q^8(q^2+1)^2(q^4+q^2+1)/8>q^{16}/8$. Lemmas~\ref{lem:New} and \ref{lem:six}(c) imply that $r$ divides $6a\gcd(4,q^{4}-1)\cdot |H_0|$.
	But $\gcd(v-1,q^2-1)\leq 2$ and $4$ dose not divide $v-1$. Then $r$ divides $6a_{2'}$. By Lemma~\ref{lem:six}(c) and the fact that  $a^2\leq 2q$, we have that $q^{15} < 576$, which is impossible.  If $\e'=-$, then $|H_0|=2^cq^4(q^2+1)^2(q^2-1)^2/\gcd(4,q^{4}-1)$ with $c=2,3$, and $v\geq q^8(q^6-1)(q^2-1)/8>q^{16}/32$ and $v$ is even and is divisible by $q^2-1$. So $r$ is a divisor of the odd part of $3a(q^2+1)^2$. Here by Lemma~\ref{lem:six}(c), $q^8<2^7{\cdot}3^2a^2$, whence $q=2$. But then $r\leq 75$, and so $6048 \leq  v<r^2\leq 75^{2}=5625$, which is impossible.
	Therefore, $m\geq5$, but this case can be ruled out by the same argument as in the case where $H=N_i^\delta$, that is to say, part (1) on page 20. In the latter case where $(X,H_0)=(\POm_{2m}^{\pm}(q),\rm{\Omega}_{m}(q)^2{\cdot}4)$ with $mq$ odd, by \eqref{eq:v}, $v>q^{m^2-5}$, and by \cite[p. 333]{a:Saxl2002}, $r<4aq^m$, and this violates the fact that $\lambda v<r^2$.\smallskip
	
	\noindent{(b)}\quad
	The possibilities for $(X,H_{0})$ in this case are listed in Table~\ref{tbl:orth-c6-s}, and for each such $H_{0}$, by~\eqref{eq:v}, and the fact that $r$ divides $\gcd(|H|,v-1)$, we obtain $l_v$ and $u_r$ as in the same table, but then for each possibilities, $\lambda v<r^2$ does not hold, which is a contradiction.\smallskip
	
	\noindent{(c)}\quad  In this case, by~\cite[Proposition 4.2.15]{b:KL-90}, and since $|G|<|H|{\cdot}|H|^2_{p'}$, the only possibilities are shown in Table~\ref{tbl:orth-c6-s}. Here $H_0=2^{n-1}\A_n$ or $2^{n-2}\A_n$, respectively for $n$ is odd or even as in Table~\ref{tbl:orth-c6-s}, and these cases can also be ruled out as $r^2<v$.\smallskip
	
	\noindent\textbf{(3)} Let $H$ be a $\Cmc_2$-subgroup of type $\GL_{m}(q)$, where $m=n/2$. Thus by~\cite[Proposition 4.2.7]{b:KL-90}, $\e=+$ and $H_0\cong \, ^{\hat{}}\SL_{m}(q){\cdot}(q-1){\cdot}\gcd(m,2)/\gcd(q-1,2)$. So by~\eqref{eq:v}, we obtain $v>q^{m(m-1)}/2$. Assume first that $m=4$. Then by \cite[p. 333]{a:Saxl2002}, $r<4aq^4$ if $q$ is odd and $r\leq aq^4$ if $q$ is even. In both cases, $r^2<v$, which is a contradiction.
	Assume now that $m\geq5$. Then Lemmas~\ref{lem:subdeg-PSp} and \ref{lem:six}(d) imply that $r$ divides the odd part of $a(q^m-1)(q^{m-1}-1)/(q+1)$ as $\gcd(r,v)=1$. Thus Lemma~\ref{lem:six}(d) yields $q^{m(m-1)}/2<v<r^2<a^2q^{4m-2}$, and so $q^{m^2-5m+2}<2a^2$, which is impossible. \smallskip
	
	\noindent\textbf{(4)} Here, by \cite[p. 333]{a:Saxl2002}, if $q\geq 3$, then $G$ has a subdegree dividing $n(n-2)(q+1)^{2}|\Out(X)|$, and if $q=2$, we have a subdegree dividing $n(n-2)(n-4)(q+1)^{3}|\Out(X)|$. Both cases violate $\lambda v<r^2$.\smallskip
	
	\noindent\textbf{(5)} Let $H$ be a $\Cmc_3$-subgroup of type $\O_{m}^{\e'}(q^2)$ or $\GU_{m}(q)$ with $m=n/2$. Then by \cite[Propositions 4.3.14-4.3.18 and 4.3.20]{b:KL-90}, one of the following holds:
	\begin{enumerate}[(a)]
		\item $H\cong N_{G}(\rm{\Omega}_{m}^{\e'}(q^2))$ with $\e'=\pm$ if $m$ is even and empty otherwise;
		\item $H\cong N_{G}(^{\hat{}}\GU_{m}(q))$ with $\e=(-1)^{m}$.
	\end{enumerate}
	
	Assume first that $H\cong N_{G}(\rm{\Omega}_{m}^{\e'}(q^2))$ with $\e'=\pm$ if $m$ is even and empty otherwise. If $q$ is odd, then we apply Lemma~\ref{lem:divisible} to $H$ and conclude that an index of a parabolic subgroup of $H$
	divides $r$, and it follows from Lemma~\ref{lem:odd-deg} that $r$ is even, but we know that $v$ is even and this contradicts the fact that $r$ divides $v-1$.
	Hence, $q$ is even and therefore $m$ is also even. Then~\cite[Propositions 4.3.14 and 4.3.16]{b:KL-90} and \eqref{eq:v} imply that
	$v=q^{\frac{1}{2}m^2}(q^{2m-2}-1)(q^{2m-6}-1)\cdots (q^2-1)/2^c$ with $c=1,2$. Note that $v$ is even and $r$ is odd. Also by \cite[Lemma 4.2 and Corollary 4.3]{a:AB-Large-15}, we have that $v>q^{m^2-5}/4$. As $r$ divides $v-1$, it is coprime to $q^2-1$. Here $|H_0|=2cq^{m(m/2-1)}(q^{m}-\e') \prod_{i=1}^{m/2-1}(q^{4i}-1)/\gcd(4,q^{m/2}-\e')$ with $c=1,2$, and so
	$r\leq6aq^{(m^2-3m+2)/2}$. Thus Lemma~\ref{lem:six}(c) yields $q^{m^2-5}/4<v<r^2\leq36a^2q^{m^2-3m+2}$, and so $q^{3m-7}<144a^2$. This inequality holds only for $m=4$ and $q=2$. But then $v\geq 12096$ and $r\leq 48$, and so $v>r^2$, which is a contradiction.\smallskip
	
	Assume finally that $H\cong N_{G}\left(^{\hat{}}\GU_{m}(q)\right)$ with $\e=(-1)^{m}$. Then by \eqref{eq:v}, we have that $v>q^{m(m-1)}/2$. If $q$ is odd, we argue as before. Let $q$ be even. If $m=4$, then by \cite[p. 334]{a:Saxl2002}, $r\leq a_{2'}(q+1)(q^3+1)$. But then we have that $r^2<v$, which is a contradiction.
	Thus $m\geq5$. Here by Lemmas \ref{lem:subdeg-PSp} and \ref{lem:six}(d), $r$ divides $a(q^m-(-1)^m))(q^{m-1}-(-1)^{m-1})$, and so $r<2aq^{2m-1}$. Therefore, by Lemma~\ref{lem:six}(c), we conclude that $q^{m^2-5m+2}< 8a^2$, and this inequality holds only for $m=5$ and $q \in \{2,4,8\}$. But then  $q^{20}/2<v<r^2\leq a^2_{2'}(q^5+1)^2(q^4-1)^2$, and so $q^{20}<2a^2_{2'}(q^5+1)^2(q^4-1)^2$, which does not hold.\smallskip
	
	\noindent\textbf{(6)} Let $H$ be a $\Cmc_4$-subgroup of type $\Sp_{n/2}(q)\otimes \Sp_{2}(q)$ and $(n, \e) \in\{(12, +), (8, +)\}$. Then
	$(X, H_0)$ is $(\POm_{8}^{+}(q),\PSp_{4}(q){\times}\PSp_{2}(q))$ or $(\POm_{12}^{+}(q),\PSp_{6}(q){\times}\PSp_{2}(q))$.  If $X=\POm_{8}^{+}(q)$, then $v=q^7(q^6-1)(q^2+1)$, and again by \cite[p. 334]{a:Saxl2002}, $r<4aq^4$. Thus by Lemma~\ref{lem:six}(c), $q^{15}/2<v<r^2<16a^2q^8$, and so $q^6<2^6$, which is a contradiction.
	If $X=\POm_{12}^{+}(q)$, then we apply Corollary~\ref{cor:large-2},  \cite[Corollary 4.3]{a:AB-Large-15} and Lemma~\ref{lem:equ}, and since $a^2\leq2q$, we deduce that $q^{66}/8<2^7q^{53}$, and so  $q^{13}<2^{10}$, which is impossible.\smallskip
	
	\noindent\textbf{(7)} Let $H$ be a $\Cmc_5$-subgroup of type $\O_{n}^{\e'}(q_{0})$ with $q=q_{0}^2$. If $n$ is odd, then by \cite[Proposition 4.5.8]{b:KL-90}, $H_0=\rm{\Omega}_n(q_0){\cdot}2$, and $v$ is odd. As $\gcd(r,p)=1$, Lemmas~\ref{lem:divisible} and \ref{lem:odd-deg} imply that $r$ is even, which is a contradiction. If $n$ is even, then by \cite[Proposition 4.5.10]{b:KL-90}, $H_0=\POm_{n}^{\e'}(q_{0}){{\cdot}}2^{c}$ with $c\leq2$ and $n=2m$, and by \eqref{eq:v}, $v>q_{0}^{m(2m-1)}/4$. We apply Lemmas~\ref{lem:subdeg-PSp} and \ref{lem:six}(d) and conclude that $r<16aq_0^{2m-1}$ for $m\geq5$ and $r<48aq_0^{2m-1}$ for $m=4$. If $m\geq5$, then by Lemma~\ref{lem:six}(c) and the fact that  $a^2\leq2q=2q_0^2$, we have that $q_0^{2m^2-5m}<2^{11}$, that is to say, $2m^2-5m-11<0$, which is impossible. If $m=4$, then $q_0^{14}<2^{10}{\cdot}3^2a^2$, and so $q=q_{0}^{2}=4$. Thus $H_{0}=\rm{\Omega}_{8}^{-}(2)$. But then $2^{26}<v$ and $r\leq2^{2}{\cdot}{3}{\cdot}{7}{\cdot}17$, which is a contradiction. \smallskip
	
	To complete the proof, we only need to discuss the case where $X=\POm^+_8(q)$ and $G$ contains a triality automorphism  when $H_0$ is a parabolic subgroup, $\G_2(q)$ or $[2^{9}]{\cdot}\PSL_{3}(2)$  for $q=3$. If $H_0$ is a parabolic subgroup of $X$, then  according to \cite[p. 335]{a:Saxl2002}, it is either $P_2$ or $P_{134}$. The first case was ruled out in (1). For the latter case, $q^{11}<v<r^2\leq 9q^2$, which is impossible.
	This leaves only to consider the cases where $H_{0}$ is $\G_2(q)$ or $[2^{9}]{\cdot}\PSL_{3}(2)$  for $q=3$. In the former case, as $q+1$ divides both $r$ and $v$, which is a contradiction. In the latter case, $3^{11}{\cdot}5^{2}{\cdot}13=v<r^2\leq2^2{\cdot}7^2$, which is impossible.
\end{proof}

\subsection{Proof of Theorem~\ref{thm:main}}

The proof follows from Propositions~\ref{prop:psl2}-\ref{prop:orth} and the main results in \cite{a:A-Exp-CP,a:ABD-Un-CP,a:Zhou-sym-sporadic,a:Zhan-CP-nonsym-sprodic,a:Zhou-CP-nonsym-alt,a:Zhuo-CP-sym-alt,a:Saxl2002}. More precisely, the $2$-designs in Examples \ref{ex:witt}-\ref{ex:suzuki} arise from studying $2$-designs in \cite{a:A-Exp-CP,a:ABD-Exp,a:Saxl2002} when the socle of $G$ is a finite simple exceptional group, and the $2$-designs in Example \ref{ex:other} are obtained in \cite{a:ABD-Un-CP,a:Zhou-sym-sporadic,a:Zhan-CP-nonsym-sprodic,a:Zhou-CP-nonsym-alt,a:Zhuo-CP-sym-alt}. For the remaining possibilities of the socle $X$ of $G$, by Propositions~\ref{prop:psl2}-\ref{prop:orth}, the result follows. 

\subsection{Proof of Corollary~\ref{cor:rp}}

Suppose that $\Dmc$ is a a nontrivial $2$-design with prime replication number $r$. Suppose also that $G$ is a flag-transitive automorphism group of $\Dmc$ of almost simple type with socle $X$, and  $\alpha\in \Pmc$. Then  by \cite{a:ABCD-PrimeRep}, $\Dmc$ and $G$ are as in lines {\rm 1-3, 7, 11, 13} of {\rm Table~\ref{tbl:main}} or $\Dmc$ is the Witt-Bose-Shrikhande space $\W(2^a)$ as in part (b), or $X=\PSL_{n}(q)$ with $n\geq 3$, $G_{\alpha}\cap X=\,^{\hat{}}[q^{n-1}]{:}\SL_{n-1}(q){\cdot} (q-1)$, $v=(q^{n}-1)/(q-1)$ and $r$ is a primitive divisor of $(q^{n-1}-1)/(q-1)$ . In the latter case, if $\Dmc$ is not a projective space as in part (a), then since $r$ is prime, it is coprime to $\lambda$, and hence Theorem~\ref{thm:main} implies that $\Dmc$ is a $2$-design with $r=q(q^{n-1}-1)/(q-1)$, however this case cannot occur as $r$ is prime. 

\subsection{Proof of Corollary~\ref{cor:main}}

By Proposition~\ref{prop:flag}, the group $G$ is primitive of almost simple or affine type. The latter case has been treated by Biliotti and Montinaro in \cite{a:Biliotti-CP-sym-affine} and in conclusion $G\leq \AGaL_{1}(q)$ is point-primitive and block-primitive with $q$ an odd prime power. In the case where, $G$ is of almost simple type, we apply Theorem~\ref{thm:main}. If $X=\PSL_{n}(q)$, then Propositions~\ref{prop:psl2} and \ref{prop:psl} implies that $\Dmc$ is a projective space $\PG_{n-1}(q)$ as in {\rm Example~\ref{ex:proj-space}}, and among other possible $2$-designs in Theorem~\ref{thm:main}, as mentioned in Section~\ref{sec:examples}, the only remaining symmetric design is the unique Hadamard design with parameters  $(11,5,2)$ as in line $6$ of {\rm Table~\ref{tbl:main}}. Recall that the design in line $12$ of Table \ref{tbl:main} is viewed as a projective space and is included in Example~\ref{ex:proj-space}. 

\section*{acknowledgements}
	The authors would like to thank anonymous referees for providing us helpful and constructive comments and suggestions to improve our main theorem, in particular, Proposition~4.3 and its  proof.

%
%


\end{document}